\documentclass[]{interact}

\usepackage{epstopdf}
\usepackage[caption=false]{subfig}

\usepackage{lmodern}
\usepackage{url}

\usepackage[numbers,sort&compress]{natbib}
\bibpunct[, ]{[}{]}{,}{n}{,}{,}
\makeatletter
\def\NAT@def@citea{\def\@citea{\NAT@separator}}
\makeatother

\theoremstyle{plain}
\newtheorem{theorem}{Theorem}[section]
\newtheorem{lemma}[theorem]{Lemma}
\newtheorem{corollary}[theorem]{Corollary}
\newtheorem{proposition}[theorem]{Proposition}

\theoremstyle{definition}
\newtheorem{definition}[theorem]{Definition}

\theoremstyle{remark}
\newtheorem{remark}{Remark}

\newtheorem{assumption}{Assumption}

\begin{document}
	
	
\title{Twice epi-differentiablity and parabolic regularity of a class of non-amenable functions}
	
\author{
 \name{Yulan Liu \textsuperscript{a} and Shaohua Pan\textsuperscript{b}\thanks{Corresponding author: Shaohua Pan (shhpan@scut.edu.cn)}}
 \affil{\textsuperscript{a}School of Mathematics and Statistics, Guangdong University of Technology, Guangzhou; \textsuperscript{b}School of Mathematics, South China University of Technology, Guangzhou, China}
 }
	
 \maketitle
	
 \begin{abstract}
 This paper concerns the twice epi-differentiability and parabolic regularity of a class of non-amenable functions, the composition of a piecewise twice differentiable (PWTD) function and a parabolically semidifferentiable mapping. Such composite functions often appear in composite optimization problems, disjunctive optimization problems, and low-rank and/or sparsity optimization problems. By establishing the proper twice epi-differentiability and parabolic epi-differentiability of PWTD functions, we prove the parabolic epi-differentiability of this class of composite functions, and its twice epi-differentiability under the parabolic regularity assumption. Then, we identify a condition to ensure its parabolic regularity with the help of an upper and lower estimate of its second subderivative, and demonstrate that this condition holds for several classes of specific non-amenable functions.
 
\end{abstract}
	
	
\begin{amscode}
 49J52; 90C46; 46G05
\end{amscode}

\section{Introduction}
Let $\mathbb{X}$ represent a finite dimensional real vector space endowed with the inner product $\langle \cdot,\cdot\rangle$ and its induced norm $\|\cdot\|$. We are interested in the composite function
 \begin{equation}\label{ffun}
 f(x):=\vartheta(F(x))\quad{\rm for}\ \ x\in\mathbb{X},	
 \end{equation}
 where $\vartheta\!:\mathbb{R}^m\to\overline{\mathbb{R}}\!:=\!\mathbb{R}\cup \{\infty\}$ and $F\!:\mathbb{X}\to\mathbb{R}^m$ satisfy the following basic assumption:
 \begin{assumption}\label{ass0}
 \begin{description}
 \item[(i)] the function $\vartheta$ is piecewise twice differentiable (PWTD), and is strictly continuous relative to its domain ${\rm dom}\,\vartheta\ne\emptyset$;  
 
 \item[(ii)] the mapping $F$ is parabolically semidifferentiable on an open set $\mathcal{O}\!\supset\! F^{-1}({\rm dom}\,\vartheta)$;   
 \end{description}
 \end{assumption} 
 \noindent
 Here we adopt ``strictly continuous'' rather than ``locally Lipschitz continuous'' because this paper adopts the notation as used in \cite{RW98}. According to  Definition \ref{Def-PTD} below, we assume that ${\rm dom}\,\vartheta=\bigcup_{i=1}^sC_i$ for polyhedral sets $C_1,\ldots,C_{s}$, and on each $C_i$, $\vartheta$ equals a function $\vartheta_i$ that is twice differentiable on an open superset of $C_i$.
 
 The function $f$ defined in \eqref{ffun} often appears in the composite optimization problem 
 \begin{equation}\label{prob}
 \min_{x\in\mathbb{X}}\Phi(x):=f_0(x)+f(x),
 \end{equation}
 where $f_0\!:\mathbb{X}\to\overline{\mathbb{R}}$ is a lower semicontinuous (lsc) function that is twice differentiable on an open set containing $F^{-1}({\rm dom}\,\vartheta)$. Such a composite optimization problem covers major classes of constrained optimization problems such as classical nonlinear programming, second-order cone and semidefinite programming (see \cite{BS00}), eigenvalue optimization problems (see \cite{Overton96,Song15,Torki99}), and amenable composite optimization problems \cite{Roc89}. As the outer $\vartheta$ is not required to be convex, some new problems are also absorbed in model \eqref{prob} such as disjunctive programs \cite{Balas18,Gfrerer14} and composite problems arising from low-rank and/or sparsity optimization. For example, the loss function $\psi(\mathcal{A}(UV^{\top}\!))$ involved in the  factorization form of low-rank optimization problems precisely has the form of \eqref{ffun}, where $\psi\!:\mathbb{R}^m\to\mathbb{R}$ is a piecewise linear-quadratic (PWLQ) function such as the SCAD or MCP function (see \cite{Fan01,Zhang10}) and $\mathcal{A}\!:\mathbb{R}^{n_1\times n_2}\to\mathbb{R}^m$ is a sampling operator. In addition, the group SCAD and MCP functions in group sparsity optimization \cite{Guo15,Ogutu14}, which are shown to be the equivalent DC surrogates of group zero-norm in \cite{LiuBiPan18}, also take the form of \eqref{ffun} with a parabolically semidifferentiable $F$.
 
 Composite functions of form \eqref{ffun} constitute a convenient framework in variational analysis and continuous optimization for developing theories and algorithms of constrained and composite optimization. Standard assumptions under which this class of functions was investigated and applied in constrained optimization require that the inner mapping $F$ is twice continuously differentiable, that the outer function $\vartheta$ is lsc and convex, and that the epigraphical multifunction associated with $\vartheta$ is metrically regular around the point of of interest (see \cite{RW98}). Compared with metric regularity, metric subregularity of a multifunction	or, equivalently, calmness of its inverse mapping has more important application scenarios where its robust counterpart is out of reach. We refer the reader to the reference \cite{Aragon14,Gfrerer11,Ioffe08,Gfrerer16,Kruger,Ye97,Li12,Zheng10} for various developments on metric subregularity and calmness, and their applications to optimality conditions, error bounds, and convergence of algorithms. Inspired by this, Mohammadi et al. \cite{MohMS22-MOR} carried out the first-order and second-order variational anaysis for a class of fully subamenable functions of the form \eqref{ffun}, which are the same as fully amenable functions \cite{Roc89} except that metric regularity of epigraphical multifunctions required by the latter is replaced by metric subregularity qualification condition (MSQC) for system $F(x)\in{\rm dom}\,\vartheta$ at the point of interest. 
 
 Twice epi-differentiability of extended real-valued functions, as a carrier of vital second-order information, plays a significant role in second-order variational analysis of nonconvex and nonsmooth optimization, but its proof has been well recognized to be extremely difficult. Since Rockafellar's landmark paper \cite{Roc89} where this property was proved for the fully amenable functions, there are few works on this topic except \cite{Ioffe91,Cominetti91,Levy01}. The papers \cite{Ioffe91,Cominetti91} extended Rockafellar's results to the composition functions of form \eqref{ffun} with a proper lsc convex $\vartheta$ and a twice differentiable $F$, but required a restrictive assumption on the second subderivative that does not hold for constrained optimization problems; while the work \cite{Levy01} obtained upper and lower estimates for the second subderivative but did not discuss the twice epi-differentiability. Until recently, Mohammadi et al. \cite{MohMS21-TAMS,MohS20-SIOPT,MohMS22-MOR} conducted a systematic study for the twice epi-differentiability of this class of compositions under MSQCs by leveraging their parabolic epi-differentiability and regularity. In \cite{MohMS21-TAMS}, to achieve the twice epi-differentiability of the indicator function of the set $\Omega=g^{-1}(\Theta)$ where $g\!:\mathbb{R}^n\to\mathbb{R}^m$ is a mapping that is twice differentiable at the point of interest, and $\Theta\subset\mathbb{R}^m$ is a closed convex set, they fully analyzed parabolic regularity for constraint system $\Omega\cap\mathcal{O}=\big\{x\in\mathcal{O}\,|\,g(x)\in\Theta\}$ where $\mathcal{O}$ is a neighborhood of the point of interest. Later, they proved in \cite{MohS20-SIOPT} the twice epi-differentiability of the composite functions of form \eqref{ffun} under parabolic regularity by assuming that the outer $\vartheta$ is proper, lsc, convex, and strictly continuous relative to its domain. Benko and Mehiltz \cite{Benko22Arxiv} studied a chain rule and a marginal function rule for the second subderivative, which under the directional calmness of $F$ yield lower estimates for the second derivative of $f$ with an lsc $\vartheta$ and a continuous $F$,  but did not touch its twice epi-differentiability.

 This work aims to establish the twice epi-differentiability for a class of non-amenable functions, i.e., the composition \eqref{ffun} with  $\vartheta$ and $F$ satisfying Assumption \ref{ass0}. Although the domain of the outer $\vartheta$ possesses a certain polyhedrality, the associated optimization model \eqref{prob} still involves non-polyhedral conic optimization problems because the inner $F$ is not required to be differentiable. To attain this goal, we conduct in Section \ref{sec3} a systematic study for the second-order variational properties of PWTD functions, and establish their proper twice epi-differentiability and parabolic epi-differentiability. Based on these properties of the outer $\vartheta$, in Section \ref{sec4}, we obtain the parabolic epi-differentiability of this class of non-amenable functions, and establish its proper twice epi-differentiability under a parabolic regularity assumption. In Section \ref{sec5}, we identify a condition to ensure the parabolic regularity of this class of non-amenable functions with the help of an upper and lower estimate for their second subderivatives, and demonstrate that the condition holds for several specific classes of non-amenable functions. To the best of our knowledge, this work is the first to explore the twice epi-differentiability of non-amenable functions. Although Mohammadi \cite{Moh22-archive} studied their first-order variational properties, he did not discuss their second-order properties.

 \medskip
 \noindent
 {\bf Notation.}  For an extended real-valued function $h\!:\mathbb{X}\!\to\! \overline{\mathbb{R}}$, denote its domain by  ${\rm dom}\,h:=\{x\in\mathbb{X}\,|\, h(x)<\infty\}$, its epigraph by ${\rm epi}\,h\!:=\{(x,\alpha)\in\mathbb{X}\times\mathbb{R}\ |\ h(x)\le\alpha\}$, and its conjugate by $h^*$, i.e. $h^*(x^*):=\sup_{x\in\mathbb{X}}\big\{\langle x,x^*\rangle-h(x)\big\}$. Such $h$ is proper if $h(x)>-\infty$ for all $x\in\mathbb{X}$ and ${\rm dom}\,h\ne\emptyset$. If a mapping $g\!:\mathbb{X}\to\mathbb{R}^m$ is differentiable at $\overline{x}\in\mathbb{X}$,
 $\nabla g(\overline{x})$ denotes the transpose of $g'(\overline{x})$, the Jacobian of $g$ at $\overline{x}$, and if $g$ is twice differentiable at $\overline{x}$, $\nabla^2g(\overline{x})$ denotes its second-order differential mapping at $\overline{x}$. For a closed set $S\subset\mathbb{X}$, $\delta_{S}$ denotes the indicator function of $S$, i.e., $\delta_S(x)=0$ if $x\in S$, otherwise $\delta_S(x)=\infty$; and ${\rm dist}(x,S)$ denotes the distance of $x$ from $S$ on the norm $\|\cdot\|$. Let $\mathbb{B}_{\mathbb{X}}$ denote the unit ball centered at the origin of $\mathbb{X}$, and  $\mathbb{B}(x,\varepsilon)$ denotes the closed ball of radius $\varepsilon$ centered at $x\in\mathbb{X}$ on the norm $\|\cdot\|$.   For an integer $k\ge 1$, write $[k]:=\{1,\ldots,k\}$. For every $q\in (1,\infty)$, $\|\cdot\|_q$ represents the $\ell_q$-norm of vectors in $\mathbb{R}^n$. For a closed set $C\subset\mathbb{R}^m$ and a point $y\in\mathbb{R}^m$, ${\rm dist}_{2}(y,C)$ means the distance of $y$ from $C$ on the $\ell_2$-norm. Unless otherwise stated, we always write $F(x):=(F_1(x),\ldots,F_m(x))^{\top}$, and for each $y=(y_1,\ldots,y_m)^{\top}\in\mathbb{R}^m$, define the function $(yF)\!:\mathbb{X}\to\mathbb{R}$ by $(yF)(x):=\langle y,F(x)\rangle=\sum_{i=1}^my_iF_i(x)$.

 \section{Preliminaries}
 This section includes some basic concepts on variational analysis (see the monographs \cite{RW98,Mordu18} for more details) and preliminary results that will be used later. 
 \begin{definition}
  (see \cite[Definition 8.3]{RW98}) Consider a function $h\!:\mathbb{X}\to\overline{\mathbb{R}}$ and a point $x\in{\rm dom}\,h$. The regular subdifferential of $h$ at $x$ is defined as
  \[
   \widehat{\partial}h(x):=\bigg\{v\in\mathbb{X}\ |\ 
   \liminf_{x\ne x'\to x}\frac{h(x')-h(x)-\langle v,x'-x\rangle}{\|x'-x\|}\ge 0\bigg\},
  \]
  and the basic (also known as the limiting or Morduhovich) subdifferential of $h$ at $x$ is defined as
  \[
 	\partial h(x):=\Big\{v\in\mathbb{X}\ |\ \exists\,x^k\to x\ {\rm with}\ h(x^k)\to h(x)\ {\rm and}\
 	v^k\in\widehat{\partial}h(x^k)\ {\rm s.t.}\ v^k\to v\Big\}.
 \]
 \end{definition}
 
 When $h=\delta_{S}$ for a nonempty closed set $S\subset\mathbb{X}$, its regular subdifferential becomes the regular normal cone to $S$ at $x$, denoted by $\widehat{\mathcal{N}}_{S}(x)$; and its subdifferential reduces to the normal cone to $S$ at $x$, denoted by $\mathcal{N}_{S}(x)$. Furthermore, when $\widehat{\mathcal{N}}_{S}(x)=\mathcal{N}_{S}(x)$, the set $S$ is said to be (Clarke) regular at $x$.
  
 For a multifunction $\mathcal{F}\!:\mathbb{X}\rightrightarrows\mathbb{R}^m$, we denote by  $\mathcal{F}^{-1}(y):=\{x\in \mathbb{X}\ |\ y\in \mathcal{F}(x)\}$ its inverse mapping, and by ${\rm gph}\mathcal{F}:=\{(x,y)\in \mathbb{X}\times \mathbb{R}^m\ |\ y\in \mathcal{F}(x)\}$ its graph. Recall that $\mathcal{F}$ is said to have the (metric) subregularity property with modulus $\kappa>0$ at a point $(\overline{x},\overline{y})\in{\rm gph}\mathcal{F}$ if there exists $\varepsilon>0$ such that for all $x\in\mathbb{B}(\overline{x},\varepsilon)$,
 \[
  {\rm dist}(x,\mathcal{F}^{-1}(\overline{y}))\le\kappa\,{\rm dist}_2(\overline{y},\mathcal{F}(x)).
 \] 
 To characterize the tangent cones and second-order tangent sets to ${\rm dom}f$ later, we need the MSQC for constraint system $F(x)\in{\rm dom}\,\vartheta$, a mild condition also used in \cite{Gfrerer13,Gfrerer16-SIOPT,Benko22,Ioffe08}.
 \begin{definition}
  The MSQC is said to hold for system $F(x)\in{\rm dom}\,\vartheta$ at a point $\overline{x}\in{\rm dom}\,f$ if the multifunction $\mathcal{F}(x)\!:=F(x)-{\rm dom}\,\vartheta$ is metrically subregular at $(\overline{x},0)$.
 \end{definition} 
 \subsection{Subderivative and tangent cones}\label{sec2.1}
 
 Before stating the formal definition of subderivative of a function, we recall the epi-convergence of a sequence of functions $\{h^k\}_{k\in\mathbb{N}}$ on $\mathbb{X}$. 
 \begin{definition}(see \cite[Definition 7.1]{RW98})
 For any sequence $\{h^k\}_{k\in\mathbb{N}}$ of functions on $\mathbb{X}$,
  the lower epi-limit, denoted by e-$\liminf_{k}h^{k}$, is the function having $\limsup_{k}({\rm epi}\,h^{k})$ as its epigraph, i.e.,
 ${\rm epi}\,(\textrm{e-}{\textstyle\liminf_{k}}h^{k}\big)
  :={\textstyle\limsup_{k}}({\rm epi}\,h^{k})$; and the upper epi-limit, denoted by e-$\limsup_{k}h^{k}$, is the function having $\liminf_{k}({\rm epi}\,h^{k})$ as its epigraph, i.e., ${\rm epi}\,(\textrm{e-}{\textstyle\limsup_{k}}h^{k}\big):={\textstyle\liminf_{k}}({\rm epi}\,h^{k})$. When these two functions coincide, the epi-limit function e-$\lim_{k}h^{k}$ is said to exist with $h:=\textrm{e-}{\textstyle\lim_{k}}h^{k}:=\textrm{e-}{\textstyle\limsup_{k}}h^{k}=\textrm{e-}{\textstyle\liminf_{k}}h^{k}$. In this event, $h^{k}$ is said to epi-converge to $h$. Clearly, $h^{k}\xrightarrow[]{e}h$ if and only if $ 	{\rm epi}\,h^{k}\to {\rm epi}\,h$.
  \end{definition}
 \begin{definition}\label{sderive-def}
  (see \cite[Definitions 8.1 \& 7.20]{RW98}) Consider a function $h\!:\mathbb{X}\to\mathbb{\overline{R}}$ and a point $x\in{\rm dom}\,h$. Let $\Delta_{\tau}h(x)\!:\mathbb{X}\to\overline{\mathbb{R}}$ be the first-order difference quotients of $h$ at $x$:
  \[
 	\Delta_{\tau}h(x)(w'):=\tau^{-1}\big[h(x+\tau w')-h(x)\big]\quad{\rm for}\ \tau>0.
  \] 
  The subderivative function $dh(x)\!:\mathbb{X}\to[-\infty,\infty]$ of $h$ at $x$ is defined as
  \[
 	dh(x)(w):=\liminf_{\tau\downarrow 0,w'\to w}\Delta_{\tau}h(x)(w')\quad{\rm for}\ w\in \mathbb{X}.
  \]
  The function $h$ is called (properly) epi-differentiable at $x$ if $\Delta_{\tau}h(x)$ epi-converge to the (proper) function $dh(x)$ as $\tau\downarrow 0$. The function $h$ is said to be semidifferentiable at $x$ for $w$ if the limit $\lim_{\tau\downarrow 0,w'\to w} \Delta_{\tau}h(x)(w')$ exists and is finite, and it is the semiderivative of $h$ at $x$ for $w$; and if this holds for any $w$, $h$ is semidifferentiable at $x$.
 \end{definition}
 
   We say that a mapping $g\!:\mathbb{X}\to\mathbb{R}^m$ is semidifferentiable at $x$ for $w\in\mathbb{X}$ if its each component function $g_i\!:\mathbb{X}\to\mathbb{R}$ is semidifferentiable at $x$ for $w$, and now we denote by $dg(x)(w):=(dg_1(x)(w),\ldots,dg_m(x)(w))^{\top}$ the semiderivative of $g$ at $x$ for $w$. The mapping $g$ is said to be semidifferentiable at $x$ if it is semidifferentiable at $x$ for every $w\in\mathbb{X}$. By \cite[Theorem 7.21 (e)]{RW98}, one can prove that the semidifferentiability of $g$ at $x$ implies the camlness of $g$ at $x$, so the calmness of $g$ at $x$ in direction $w$ \cite[Section 2.1]{Benko22Arxiv}.
 
 With the subderivative function, we follow the same line as in \cite{MohMS21-TAMS} to introduce the critical cone to an extended real-valued $h\!:\mathbb{X}\to\overline{\mathbb{R}}$ at a point $(x,v)\in{\rm gph}\,\partial h$:
 \begin{equation*}
 \mathcal{C}_{h}(x,v):=\big\{w\in\mathbb{X}\ |\ dh(x)(w)=\langle v,w\rangle\big\}.
 \end{equation*}  
  When $h=\delta_{S}$ for a nonempty closed set $S\subset\mathbb{X}$, the subderivative $dh(x)$ for $x\in{\rm dom}\,h$ is the indicator of $\mathcal{T}_{S}(x)$, the tangent cone to $S$ at $x$, and now $\mathcal{C}_{h}(x,v)=\mathcal{T}_{S}(x)\cap\{v\}^{\perp}$. The tangent and inner tangent cones to $S$ at $x\in S$ are respectively defined as
 \begin{align*}
 \mathcal{T}_{S}(x)&:=\big\{w\in\mathbb{X}\ |\ \exists\, \tau_k\downarrow 0\ \ {\rm s.t.}\ \ {\rm dist}(x+\tau_kw,S)=o(\tau_k)\big\},\\
 \mathcal{T}_{S}^{i}(x)&:=\big\{w\in\mathbb{X}\ |\ {\rm dist}(x+\tau w,S)=o(\tau)\ \ \forall\tau\ge 0\big\}.
 \end{align*}
 We stipulate that $\mathcal{T}_{S}(x)\!=\!\mathcal{T}^i_{S}(x)\!=\!\emptyset$ if $x\!\notin\! S$. 
 The following lemma characterizes the tangent cone to ${\rm dom}\,f$, where $dF(\overline{x})(w)$ is well defined because the parabolic semidifferentiability of $F$ by Assumption \ref{ass0} implies its semidifferentiability by Definition \ref{psemi-deriv}.
 \begin{lemma}\label{tcone-domf}
 Let $\overline{x}\in{\rm dom}\,f$. If the MSQC holds for system $F(x)\!\in\!{\rm dom}\vartheta$ at $\overline{x}$, then
  \begin{align}\label{equa-tcone}
  \mathcal{T}_{{\rm dom}\,f}(\overline{x})
   &=\big\{w\in\mathbb{X}\ |\ dF(\overline{x})(w)\in\mathcal{T}_{{\rm dom}\,\vartheta}(F(\overline{x}))\big\}\nonumber\\
   &=\Big\{w\in\mathbb{X}\ |\ dF(\overline{x})(w)\in\bigcup\limits_{i=1}^s\mathcal{T}_{C_i}(F(\overline{x}))\Big\}=\mathcal{T}_{{\rm dom}\,f}^{i}(\overline{x}).
  \end{align} 	
 \end{lemma}
 \begin{proof}
 Since the MSQC holds for system $F(x)\!\in\!{\rm dom}\,\vartheta$ at $\overline{x}$, the first equality is implied by \cite[Proposition 1]{Herion05}. From ${\rm dom}\,\vartheta\!=\!\bigcup_{i=1}^s C_i$ and \cite[Proposition 3.37]{BS00}, it follows that
  \[
   \bigcup\limits_{i=1}^s\mathcal{T}_{C_i}^{i}(F(\overline{x}))
    \subset\mathcal{T}_{{\rm dom}\,\vartheta}^i(F(\overline{x}))
     \subset\mathcal{T}_{{\rm dom}\,\vartheta}(F(\overline{x}))
    =\bigcup\limits_{i=1}^s\mathcal{T}_{C_i}(F(\overline{x}))
   =\bigcup\limits_{i=1}^s\mathcal{T}_{C_i}^{i}(F(\overline{x})),
  \]
  which implies that the second equality holds and $\mathcal{T}_{{\rm dom}\,\vartheta}^i(F(\overline{x}))=\mathcal{T}_{{\rm dom}\,\vartheta}(F(\overline{x}))$. Along with the first equality in \eqref{equa-tcone} and $\mathcal{T}_{{\rm dom}f}^{i}(\overline{x})\subset\mathcal{T}_{{\rm dom}f}(\overline{x})$, the rest only needs to prove that
  \begin{equation}\label{aim-equa-tcone}
  \big\{w\in\mathbb{X}\ |\ dF(\overline{x})(w)\in\mathcal{T}^i_{{\rm dom}\,\vartheta}(F(\overline{x}))\big\}\subset\mathcal{T}_{{\rm dom}\,f}^{i}(\overline{x}).
  \end{equation}
  Pick any $w\in\mathbb{X}$ with $dF(\overline{x})(w)\in\mathcal{T}^i_{{\rm dom}\,\vartheta}(F(\overline{x}))$. By the definition of the inner tangent cone, for any sufficiently small $\tau>0$, ${\rm dist}_2(F(\overline{x})\!+\!\tau dF(\overline{x})(w),{\rm dom}\,\vartheta)=o(\tau)$. Recall that the mapping $F$ is semidifferentiable at $\overline{x}$ for $w$. By Definition \ref{sderive-def}, $F(\overline{x}+\tau w)-F(\overline{x})-\tau dF(\overline{x})(w)=o(\tau)$ and then ${\rm dist}_2(F(\overline{x}+\tau w),{\rm dom}\,\vartheta)=o(\tau)$. In addition, since the MSQC holds for system $F(x)\in{\rm dom}\,\vartheta$ at $\overline{x}$, for each sufficiently small $\tau>0$, there exist $x_{\tau}\in{\rm dom}\,f$ and $\kappa>0$ such that $\|\overline{x}+\tau w-x_{\tau}\|\le \kappa{\rm dist}_2(F(\overline{x}\!+\!\tau w),{\rm dom}\,\vartheta)$. The two sides show that ${\rm dist}(\overline{x}\!+\!\tau w,{\rm dom}\,f)=o(\tau)$ and $w\in \mathcal{T}^{i}_{{\rm dom}\,f}(\overline{x})$. Consequently, the inclusion in \eqref{aim-equa-tcone} holds.
\end{proof}
 
 According to \cite[Exercise 8.4]{RW98}, for a function $h\!:\mathbb{X}\to\mathbb{\overline{R}}$ and a point $x\in{\rm dom}\,h$, there is a close relation between its subderivative at $x$ and its regular subdifferential at $x$, i.e., 
 \begin{equation}\label{Rsdiff-sderiv}
  v\in\widehat{\partial}h(x)\ \Longleftrightarrow\ dh(x)(w)\ge\langle v,w\rangle\quad {\rm for\ all}\ w\in\mathbb{X}.
 \end{equation}
 
  \subsection{Second and parabolic subderivatives} 
 
 To introduce two kinds of second subderivatives for a function $h\!:\mathbb{X}\to\mathbb{\overline{R}}$ at a point $x\in{\rm dom}\,h$, we need two types of second-order difference quotients for $h$ at $x$: 
  \begin{align*}
  \Delta^2_{\tau}h(x)(w')
  :=\frac{h(x\!+\!\tau w')-h(x)-\tau dh(x)(w')}{\tau^2/2}\quad{\rm for}\ \tau>0,\\
  \Delta^2_{\tau}h(x|v)(w')
  :=\frac{h(x\!+\!\tau w')-h(x)-\tau\langle v,w'\rangle}{\tau^2/2}\quad{\rm for}\ \tau>0,
 \end{align*}
 where $\Delta^2_{\tau}h(x)(w'):=\infty$ whenever $h(x\!+\!\tau w')=dh(x)(w')=\infty$ or $-\infty$. 
 \begin{definition}\label{ssderiv-def}
 (see \cite[Definitions 13.3 \& 13.6]{RW98}) Consider a proper function $h\!:\mathbb{X}\to\mathbb{\overline{R}}$, a point $x\in{\rm dom}\,h$ and a vector $v\in\mathbb{X}$. The second subderivative of $h$ at $x$ for $v$ and $w$ is defined as
 \[
   d^2h(x|v)(w):=\liminf_{\tau\downarrow 0,w'\to w}\Delta^2_{\tau}h(x|v)(w'),
 \]
 while the second subderivative of $h$ at $x$ for $w$ (without mention of $v$) is defined as
 \[
  d^2h(x)(w):=\liminf_{\tau\downarrow 0,w'\to w}\Delta^2_{\tau}h(x)(w').
 \]
 The function $h$ is said to be (properly) twice epi-differentiable at $x$ for $v$ if $\Delta^2_{\tau}h(x|v)$ epi-converge to the (proper) function $d^2h(x|v)$ as $\tau\downarrow 0$.
 \end{definition}

 By Definition \ref{ssderiv-def} and \cite[Proposition 7.2]{RW98}, the twice epi-differentiability of $h$ at $x$ for $v$ can be equivalently stated as follows: for every $w\in\mathbb{X}$ and 
 every sequence $\tau_k\downarrow 0$ there exists a sequence $w^k\to w$ such that
 $\Delta^2_{\tau_k}h(x|v)(w^k)\to d^2h(x|v)(w)$ as $k\to\infty$.
 By \cite[Proposition 13.5]{RW98},  the function $d^2h(x|v)$ is lsc and positively homogenous of degree $2$; and if it is proper, then the inclusion ${\rm dom}\,d^2h(x|v)\subset\mathcal{C}_{h}(x,v)$ holds.
 
 Next we recall the parabolic epi-differentiability of an extended real-valued function, which plays a crucial role in characterizing the expression of its second subderivative. 
 \begin{definition}\label{psderiv-def}(\cite[Definition 13.59]{RW98})
  Consider a proper function $h\!:\mathbb{X}\to\mathbb{\overline{R}}$, a point $x\in{\rm dom}\,h$ and a vector $w\in\mathbb{X}$ with $dh(x)(w)$ finite. Define the parabolic difference quotients of $h$ at $x$ for $w$ by 
  \[
   \Delta^2_{\tau}h(x)(w|z'):=\frac{h(x+\tau w+\frac{1}{2}\tau^2z')-h(x)-\tau dh(x)(w)}{\tau^2/2}\quad{\rm for}\ \tau>0.
  \]  
  The parabolic subderivative of $h$ at $x$ for $w$ with respect to (w.r.t.) $z$ is defined as
  \[
   d^2h(x)(w|z):=\liminf_{\tau\downarrow 0,z'\to z}\Delta^2_{\tau}h(x)(w|z'),
  \]
  and $h$ is said to be parabolically epi-differentiable at $x$ for $w$ if $\Delta^2_{\tau} h(x)(w|\cdot)$ epi-converge to $d^2h(x)(w|\cdot)$ as $\tau\downarrow 0$ and ${\rm dom}\,d^2h(x)(w|\cdot)\neq \emptyset$. 
 \end{definition} 

 By Definition \ref{psderiv-def} and \cite[Proposition 7.2]{RW98}, the parabolic epi-differentiability of $h$ at $x$ for $w$ can be equivalently described as follows: ${\rm dom}\,d^2h(x)(w|\cdot)\ne\emptyset$ and for every $z\in\mathbb{X}$ and every sequence $\tau_k\downarrow 0$ there exists a sequence $z^k\to z$ such that
 \begin{equation*}
  \Delta^2_{\tau_k}h(x)(w|z^k)\to d^2h(x)(w|z)\ \ {\rm as}\ k\to\infty.
 \end{equation*}

 When $h=\delta_{S}$ for a nonempty closed set $S\subset\mathbb{X}$, for any $x\in S$ and $w\in\mathcal{T}_S(x)$, $d^2h(x)(w|\cdot)$ is precisely the indicator of the second-order tangent set to $S$ at $x$ for $w$:
 \[
  \mathcal{T}^{2}_{S}(x,w)
  :=\Big\{z\in\mathbb{X}\ |\ \exists\, \tau_k\downarrow 0\ \ {\rm s.t.}\ \ {\rm dist}(x+\tau_kw+(\tau_k^2/2)z,S)=o(\tau_k^2)\Big\}.
 \]
 The inner second-order tangent set to $S$ at $x\in S$ for $w$ is defined as
 \[
 \mathcal{T}^{i,2}_{S}(x,w)
 :=\Big\{z\in\mathbb{X}\ |\ {\rm dist}(x+\tau w+(\tau^2/2)z,S)=o(\tau^2)\quad \forall\tau\ge 0\Big\}.
 \]
 One can check that $\mathcal{T}^{i,2}_{S}(x,w)=\emptyset$ if  $w\notin\mathcal{T}_{S}^{i}(x)$, and $\mathcal{T}^{2}_{S}(x,w)=\emptyset$ if $w\notin\mathcal{T}_{S}(x)$. The set $S$ is called parabolically derivable at $x$ for $w\in\mathcal{T}_{S}(x)$ if $\mathcal{T}^{i,2}_{S}(x,w) =\mathcal{T}^{2}_{S}(x,w)\ne\emptyset$.
 \begin{definition}\label{psemi-deriv}
  A mapping $g\!:\mathbb{X}\to\mathbb{R}^m$ is said to be parabolically semidifferentiable at $x$ for $w\in\mathbb{X}$ if it is semidifferentiable at $x$ for $w$, and for any $z\!\in\!\mathbb{X}$ the following limit
  \[
 	\lim_{\tau\downarrow 0,z'\to z}\Delta_{\tau}^2g(x)(w|z'):=
 	\frac{g(x+\tau w+\frac{1}{2}\tau^2 z')-g(x)-\tau dg(x)(w)}{\tau^2/2}
  \]
  exists and is finite, and it is called the parabolic semiderivative of $g$ at $x$ for $w$ with respect to $z$, denoted by $g''(x;w,z)$. The mapping $g$ is parabolically semidifferentiable at $x$ if it is parabolically semidifferentiable at $x$ for all $w$. 
 \end{definition} 

 When $m\!=\!1$, the parabolic semiderivative $g''(x;w,z)$, if allowing to be infinite, coincides with $d^2g(x)(w|z)$. The parabolic semidifferentiability of $g$ at $x$ for $w$ is different from its Hadamard second-order directional differentiability at $x$ in the direction $w$ introduced in  \cite[Section 2.2.3]{BS00} because the latter does not require the semidifferentiability of $g$ at $x$ for $w$. The following lemma will be used to achieve the second-order tangent sets to ${\rm dom}\,f$, which weakens the twice differentiability of inner mapping in \cite[Theorem 4.3]{MohMS21-TAMS} to be the parabolic semidifferentiability.   
 \begin{lemma}\label{MSw-lemma}
  Let $\varphi\!:\mathbb{R}^m\!\to\!\overline{\mathbb{R}}$ be a proper lsc function, and let $g\!:\mathbb{X}\!\to\!\mathbb{R}^m$ be parabolically semidifferentiable at $\overline{x}\!\in\! g^{-1}({\rm dom}\,\varphi)$. For any $w$ with $dg(\overline{x})(w)\in\mathcal{T}_{{\rm dom}\,\varphi}(g(\overline{x}))$, define
  \begin{equation*}
  \mathcal{S}_{w}(p):=\big\{u\in\mathbb{X}\ |\  g''(\overline{x};w,u)+p
  \in\mathcal{T}_{{\rm dom}\varphi}^2(g(\overline{x}),dg(\overline{x})(w))\big\}\quad{\rm for}\ p\in\mathbb{R}^m.
  \end{equation*}
  If the multifunction $\mathcal{H}(x):=g(x)-{\rm dom}\,\varphi$ is subregular at $(\overline{x},0)\in{\rm gph}\,\mathcal{H}$ with modulus $\kappa$, then 	$\mathcal{S}_{w}(p)\subset\mathcal{S}_{w}(0)+\kappa\|p\|_2\mathbb{B}_{\mathbb{X}}$ for all $p\in\mathbb{R}^m$ uniformly in $w$.
 \end{lemma}
 \begin{proof}
  Fix any $p\in\mathbb{R}^m$ and $u\in\mathcal{S}_{w}(p)$. Then $g''(\overline{x};w,u)+p\in\mathcal{T}_{{\rm dom}\,\varphi}^2(g(\overline{x})\,|\, dg(\overline{x})(w))$. By the definition of second-order tangent set, there exist $\tau_k\downarrow 0$ and $u^k\to g''(\overline{x};w,u)+p$ such that $g(\overline{x})+\tau_kdg(\overline{x})(w)+\frac{1}{2}\tau_k^2u^k\in{\rm dom}\,\varphi$ for all $k\in\mathbb{N}$. 
  For each $k\!\in\!\mathbb{N}$, let $x^k\!:=\!\overline{x}+\tau_kw+\frac{1}{2}\tau_k^2u$. The parabolic semidifferentiability of $g$ at $\overline{x}$ implies that
 \[
  g(x^k)=g(\overline{x})+\tau_kdg(\overline{x})(w) +\frac{1}{2}\tau_k^2g''(\overline{x};w,u)+o(\tau_k^2).
 \]
 From the subregularity of $\mathcal{H}$ at $(\overline{x},0)$ with modulus $\kappa$, for each sufficiently large $k$, there exists $z^k\!\in g^{-1}({\rm dom}\varphi)$ such that $\|x^k\!-\!z^k\|\le\kappa{\rm dist}_2(g(x^k),{\rm dom}\varphi)$. Then, it holds that 
  \begin{align*}
  \|x^k\!-\!z^k\|
  &\le\kappa\|g(x^k)-g(\overline{x})-\tau_kdg(\overline{x})(w)-\frac{1}{2}\tau_k^2u^k\|\\
  &\leq \frac{1}{2}\tau_k^2\kappa\|u^k-g''(\overline{x};w,u)\|_2+o(\tau_k^2)
   \le\frac{1}{2}\tau_k^2\kappa\|p\|_2+o(\tau_k^2).
  \end{align*}
  This implies that the sequence $d^k\!:=\!\frac{2(x^k-z^k)}{\tau_k^2}$ is bounded, so there is $d\!\in\!\mathbb{X}$ such that $d^k\to d$ (if necessary by taking a subsequence) with $\|d\|\!\le\! \kappa\|p\|_2$. On the other hand, for all $k$ large enough,
  \[
   g^{-1}({\rm dom}\,\varphi)\ni z^k=x^k-(\tau_k^2/2)d^k=\overline{x}+\tau_kw+(\tau_k^2/2)(u-d)+o(\tau_k^2),
  \]
  which along with the parabolic semidifferentiability of $g$ at $\overline{x}$ yields that 
  \[
   {\rm dom}\,\varphi\ni g(z^k)=g(\overline{x})+\tau_kdg(\overline{x})(w)+
 	(\tau_k^2/2)g''(\overline{x};w,u\!-\!d)+o(\tau_k^2).
  \]
  This means that $g''(\overline{x};w,u\!-\!d)\!\in\!\mathcal{T}_{{\rm dom}\,\varphi}^2(g(\overline{x}),dg(\overline{x})(w))$ or $u\!-\!d\!\in\!\mathcal{S}_{w}(0)$.	Thus, ${\rm dist}(u,\mathcal{S}_{w}(0))\le\|d\|\le\kappa\|p\|_2$. By the arbitrariness of $p$ in $\mathbb{R}^m$ and $u\in\mathcal{S}_{w}(p)$, the desired inclusion then follows. The proof is completed.
 \end{proof}
  
 The following proposition discloses the link between the second-order tangent set to ${\rm dom}(\varphi\circ g)$ and the one to ${\rm dom}\,\varphi$ for $\varphi$ and $g$ from Lemma \ref{MSw-lemma}, which extends the conclusions of \cite[Theorem 4.5]{MohMS21-TAMS},  \cite[Proposition 4.3 (ii)]{MohS20-SIOPT} and \cite[Lemma 2.5]{Mehlitz21} to the constraint system $g(x)\!\in\!{\rm dom}\,\varphi$ for a parabolically semidifferentiable rather than twice (continuously) differentiable $g$, as well as removes the Clarke regularity of ${\rm dom}\,\varphi$ required in \cite[Theorem 4.5]{MohMS21-TAMS} and \cite[Proposition 4.3]{MohS20-SIOPT}.
 \begin{proposition}\label{SOTset}
  Let $h=\varphi\circ g$ where $\varphi\!:\mathbb{R}^m\to\overline{\mathbb{R}}$ is a proper lsc function, and $g\!:\mathbb{X}\to\mathbb{R}^m$ is a mapping. Consider any $\overline{x}\in{\rm dom}\,h$. Suppose that $g$ is parabolically semidifferentiable at $\overline{x}$, and that the multifunction $\mathcal{H}(x):=g(x)-{\rm dom}\,\varphi$ is subregular at $(\overline{x},0)$ with modulus $\kappa$. Then, for any $w\in\mathcal{T}_{{\rm dom}\,h}(\overline{x})$, the following equivalence holds:
  \begin{equation}\label{aim-pderive}
  z\in\mathcal{T}_{{\rm dom}\,h}^2(\overline{x},w)
  \ \Longleftrightarrow\ g''(\overline{x};w,z)\in\mathcal{T}_{{\rm dom}\,\varphi}^2(g(\overline{x}),dg(\overline{x})(w)).
 \end{equation}
 If in addition ${\rm dom}\,\varphi$ is parabolically derivable at $g(\overline{x})$ for $dg(\overline{x})(w)$, so is ${\rm dom}\,h$ at $\overline{x}$ for $w$. 
 \end{proposition}
  \begin{proof}
  Fix any $w\in\mathcal{T}_{{\rm dom}\,h}(\overline{x})$. Pick any $z\in\mathcal{T}_{{\rm dom}\,h}^2(\overline{x},w)$. Then there exist $\tau_k\downarrow 0$ and $z^k\to z$ such that $\overline{x}+\tau_kw+\frac{1}{2}\tau_k^2z^k\in{\rm dom}\,h$ for each $k\in\mathbb{N}$. By the parabolic semidifferentiability of $g$, for all sufficiently large $k$, 
  \[
  {\rm dom}\,\varphi\ni g(\overline{x}+\tau_kw+\frac{1}{2}\tau_k^2z^k)=g(\overline{x})+\tau_kdg(\overline{x})(w)+\frac{1}{2}\tau_k^2g''(\overline{x};w,z)+o(\tau_k^2).
 \]
  This shows that $g''(\overline{x};w,z)\!\in\!\mathcal{T}_{{\rm dom}\,\varphi}^2(g(\overline{x}),dg(\overline{x})(w))$, so the implication $\Longrightarrow$ follows. For the converse, pick any $z\in\mathbb{X}$ such that $g''(\overline{x};w,z)\!\in\!\mathcal{T}_{{\rm dom}\,\varphi}^2(g(\overline{x}),dg(\overline{x})(w))$.  Then there exist $\tau_k\downarrow 0$ and $\xi^k\to g''(\overline{x};w,z)$ such that $g(\overline{x})+\tau_kdg(\overline{x})(w)+\frac{1}{2}\tau_k^2\xi^k\!\in\!{\rm dom}\,\varphi$ for each $k\in\mathbb{N}$. Write $x^k\!:=\!\overline{x}+\tau_kw+\frac{1}{2}\tau_k^2z$ for each $k\in\mathbb{N}$. By the parabolic semidifferentiability of $g$ at $\overline{x}$, we have $g(x^k)=g(\overline{x})+\tau_kdg(\overline{x})(w)+\frac{1}{2}\tau_k^2g''(\overline{x};w,z)+o(\tau_k^2)$ for all sufficiently large $k$. From the subregularity of $\mathcal{H}$ at $(\overline{x},0)$ with modulus $\kappa$, for each sufficiently large $k$, there exists $z^k\in{\rm dom}\,h$ such that
  \[
   \|x^k\!-\!z^k\|\le\kappa{\rm dist}_2(g(x^k),{\rm dom}\varphi)\le\kappa\big\|g(x^k)\!-\![g(\overline{x})+\tau_kdg(\overline{x})(w)+\frac{1}{2}\tau_k^2\xi^k]\big\|_2=o(\tau_k^2).
  \]
  This implies that ${\rm dom}\,h\ni z^k=\overline{x}+\tau_kw+\frac{1}{2}\tau_k^2(z+o(\tau_k^2)/\tau_k^2)$ for all $k$ large enough. Consequently, $z\in\mathcal{T}_{{\rm dom}\,h}^2(\overline{x},w)$, and the desired implication  holds. 
 
  Assume that ${\rm dom}\,\varphi$ is parabolically derivable. Then $\mathcal{T}_{{\rm dom}\,\varphi}^2(g(\overline{x}),dg(\overline{x})(w))\!\ne\!\emptyset$. Pick any $u\in\mathcal{T}_{{\rm dom}\,\varphi}^2(g(\overline{x}),dg(\overline{x})(w))$ and  any $z\in\mathbb{X}$. Then, $z\in\mathcal{S}_{w}(p)$ with $p:=u-g''(\overline{x};w,z)$, where $\mathcal{S}_{w}$ is the multifunction defined in Lemma \ref{MSw-lemma}. By invoking Lemma \ref{MSw-lemma}, there exists $\widetilde{z}\in\mathcal{S}_{w}(0)$ such that $\|z-\widetilde{z}\|\le\kappa\|p\|_2$. From $\widetilde{z}\in\mathcal{S}_{w}(0)$ and the equivalence in \eqref{aim-pderive}, $\widetilde{z}\in\mathcal{T}_{{\rm dom}\,h}^2(\overline{x},w)$. Consequently, $\mathcal{T}_{{\rm dom}\,h}^{2}(\overline{x},w)\ne\emptyset$. Pick any $z\in\mathcal{T}_{{\rm dom}\,h}^2(\overline{x},w)$. By the subregularity of $\mathcal{H}$ at $(\overline{x},0)$, for any sufficiently small $\tau>0$, it holds that 
  \(
  {\rm dist}(\overline{x}+\tau w+\frac{1}{2}\tau^2z,{\rm dom}\,h)\le\kappa{\rm dist}_2(g(\overline{x}+\tau w+\frac{1}{2}\tau^2z),{\rm dom}\varphi),
  \)
  or equivalently 
 \begin{equation}\label{distz0}
 {\rm dist}\Big(z,\frac{{\rm dom}\,h-\overline{x}-\tau w}{\frac{1}{2}\tau^2}\Big)
 \le\kappa{\rm dist}_2\Big(\Delta_{\tau}^2g(\overline{x})(w|z), \frac{{\rm dom}\varphi-g(\overline{x})-\tau dg(\overline{x})(w)}{\frac{1}{2}\tau^2}\Big).
 \end{equation}
 Clearly, $\Delta_{\tau}^2g(\overline{x})(w|z)\to g''(\overline{x};w,z):=\zeta$ as $\tau\downarrow 0$. Furthermore, since ${\rm dom}\,\varphi$ is parabolically derivable at $g(\overline{x})$ for $dg(\overline{x})(w)$, by invoking \cite[Corollary 4.7]{RW98}, when $\tau\downarrow 0$, 
 \[
   {\rm dist}_2\Big(\zeta,\frac{{\rm dom}\varphi-g(\overline{x})-\tau dg(\overline{x})(w)}
 	{\frac{1}{2}\tau^2}\Big)
   \ \to\ {\rm dist}_2\Big(\zeta,\mathcal{T}_{{\rm dom}\varphi}^{i,2}(g(\overline{x}),dg(\overline{x})(w))\Big)=0,
 \]
 where the equality is due to the equivalence in \eqref{aim-pderive}. Thus, the right-hand side of \eqref{distz0} tends to zero as $\tau\downarrow 0$, which implies that the left-hand side must approach to $0$ as $\tau\downarrow 0$, so $z\in\mathcal{T}_{{\rm dom}\,h}^{i,2}(\overline{x},w)$. By the arbitrariness of $z\in\mathcal{T}_{{\rm dom}\,h}^2(\overline{x},w)$, we have $\emptyset\ne\mathcal{T}_{{\rm dom}\,h}^2(\overline{x},w)=\mathcal{T}_{{\rm dom}\,h}^{i,2}(\overline{x},w)$. This shows that ${\rm dom}h$ is parabolically derivable at $\overline{x}$ for $w$. 
 \end{proof}

 For any $w\!\in\!\mathbb{X}$ with $dF(\overline{x})(w)\!\in\! \mathcal{T}_{{\rm dom}\,\vartheta}(F(\overline{x}))$, by invoking \cite[Proposition 3.37]{BS00}, 
 \begin{align*}
 \bigcup_{i=1}^s\mathcal{T}_{C_i}^{i,2}(F(\overline{x}),dF(\overline{x})(w))
  &\subset\mathcal{T}_{{\rm dom}\,\vartheta}^{i,2}(F(\overline{x}),dF(\overline{x})(w))
  \subset \mathcal{T}^2_{{\rm dom}\,\vartheta}(F(\overline{x}),dF(\overline{x})(w))\\  &=\bigcup_{i=1}^s\mathcal{T}_{C_i}^2(F(\overline{x}),dF(\overline{x})(w))=\bigcup_{i=1}^s\mathcal{T}_{C_i}^{i,2}(F(\overline{x}),dF(\overline{x})(w)),
 \end{align*}
 which implies that $\mathcal{T}_{{\rm dom}\,\vartheta}^{i,2}(F(\overline{x}),dF(\overline{x})(w))
 =\mathcal{T}^2_{{\rm dom}\,\vartheta}(F(\overline{x}),dF(\overline{x})(w))$. As ${\rm dom}\,\vartheta\ne\emptyset$ and each $C_i$ is convex polyhedral, by \cite[Page 168]{BS00}, there is an active component $C_i$ such that $0\!\in\!\mathcal{T}_{C_i}^{i,2}(F(\overline{x}),dF(\overline{x})(w))$, which means that $0\!\in\! \mathcal{T}^2_{{\rm dom}\,\vartheta}(F(\overline{x}),dF(\overline{x})(w))$. Thus, ${\rm dom}\,\vartheta$ is parabolically derivable. The parabolic derivability of ${\rm dom}\,\vartheta$ was also achieved in \cite[Lemma 2.4]{Mehlitz21}. Now invoking Proposition \ref{SOTset} leads to the following corollary.
 \begin{corollary}\label{Stcone-domf}
  Fix any $\overline{x}\in{\rm dom}\,f$. If the MSQC holds for constraint system $F(x)\in{\rm dom}\,\vartheta$ at $\overline{x}$, then ${\rm dom}\,f$ is parabolically derivable at $\overline{x}$ for any $w\in\!\mathcal{T}_{{\rm dom}\,f}(\overline{x})$, i.e., 
  \begin{equation*}
  \emptyset\ne\mathcal{T}^2_{{\rm dom}\,f}(\overline{x},w)=\big\{z\in\mathbb{X}\ |\ F''(\overline{x};w,z)\in\mathcal{T}^2_{{\rm dom}\,\vartheta}\big(F(\overline{x}),dF(\overline{x})(w)\big)\big\}
 		=\mathcal{T}^{i,2}_{{\rm dom}\,f}(\overline{x},w).
 \end{equation*}
 \end{corollary}

 Parabolic regularity of an extended real-valued function, as demonstrated in \cite{MohMS21-TAMS,MohS20-SIOPT}, is a crucial property to build the relationship between its second subderivative and its parabolic subderivative. Now we recall this important property.
 \begin{definition}\label{pregular-def}
  A function $h\!:\mathbb{X}\to\overline{\mathbb{R}}$ is parabolically regular at a point $\overline{x}\in{\rm dom}\,h$ for $\overline{v}$ if for every $w$ with $dh(\overline{x})(w)=\langle \overline{v},w\rangle$, 
  \(
   \inf\limits_{z\in\mathbb{X}}\big\{d^2 h(\overline{x})(w|z)-\langle\overline{v},z\rangle\big\}=d^2 h(\overline{x}|\overline{v})(w),
  \)
  or in other words, if for any $w\in{\rm dom}\,d^2h(\overline{x}|\overline{v})$ with $dh(\overline{x})(w)=\langle \overline{v},w\rangle$, there exist, among the sequences $\tau_k\downarrow 0$ and $w^k\to w$ with $\Delta_{\tau_k}^2h(\overline{x}|\overline{v})(w^k)\!\to d^2h(\overline{x}|\overline{v})(w)$, ones with an additional property $\limsup_{k\to\infty}\,\frac{\|w^k-w\|}{\tau_k}<\infty$. 
\end{definition}

 \section{Second-order variational properties of PWTD functions}\label{sec3}
 \begin{definition}\label{Def-PTD}
 A function $h\!:\mathbb{R}^m\to\overline{\mathbb{R}}$ is said to be PWTD if its domain is nonempty and can be represented as a union of finitely many polyhedral sets, say, ${\rm dom}\,h=\bigcup_{i=1}^s\Omega_i$ for polyhedral sets $\Omega_1,\ldots,\Omega_s$, and for each $i\in[s]$, there is a function $h_i$ that is twice differentiable on an open superset of $\Omega_i$ such that $h=h_i$ on $\Omega_i$.
 \end{definition}

 For the PWTD function $\psi$ appearing in this section, we write $\emptyset\ne{\rm dom}\,\psi:=\bigcup_{i=1}^s\Omega_i$ for polyhedral sets $\Omega_1,\ldots,\Omega_s$, and for each $i\in[s]$, let $\psi_i$ be a function that is twice differentiable on an open superset of $\Omega_i$ with $\psi=\psi_i$ on $\Omega_i$. For any given $y\in{\rm dom}\,\psi$ and $w\in\mathbb{R}^m$, we write $J_{y}:=\big\{i\in[s]\,|\, y\in\Omega_i\big\}$ and $J_{y,w}\!:=\{j\in J_{y}\,|\,w\in\mathcal{T}_{\Omega_j}(y)\}$. 
 
 First, we take a closer look at the subderivatives of PWTD functions. The following lemma provides the characterization for them and extends the results of \cite[Proposition 10.21]{RW98}. Since its proof is similar to that of \cite[Proposition 10.21]{RW98}, we do not include it. 
 \begin{lemma}\label{sderive-PTD}
  For a PWTD function $\psi\!:\mathbb{R}^m\to\overline{\mathbb{R}}$, the following assertions hold. 
  \begin{enumerate}
  \item[(i)] ${\rm dom}\,\psi$ is closed and $\psi$ is continuous relative to ${\rm dom}\,\psi$, so $\psi$ is lsc on $\mathbb{R}^m$;
		
  \item[(ii)] for each $y\in{\rm dom}\,\psi$, ${\rm dom}\,d\psi(y)=\mathcal{T}_{{\rm dom}\,\psi}(y)=\ {\textstyle\!\bigcup_{j\in J_{y}}}\mathcal{T}_{\Omega_j}(y)$; 
		
  \item[(iii)] for any $y\in{\rm dom}\,\psi$ and $w\in\mathbb{R}^m$,
		\[
		d\psi(y)(w)=\lim_{\tau\downarrow 0}\Delta_{\tau}\psi(y)(w)
		=\left\{\begin{array}{cl} 	 
		\!\!\langle\nabla \psi_k(y),w\rangle\ {\rm for\ any}\ k\in J_{y,w}&{\rm if}\ J_{y,w}\ne\emptyset,\\
			\infty &{\rm if}\ J_{y,w}=\emptyset, 
		\end{array}\right.	
	    \] 
 so $\psi$ is a properly epi-differentiable function with a proper piecewise linear $d\psi(y)$.   
 \end{enumerate}
\end{lemma}

Now we characterize the second subderivatives of PWTD functions and prove their  twice epi-differentiablility, extending \cite[Proposition 13.9]{RW98} for convex PWLQ functions and \cite[Theorem 3.2]{Aze98} for convex PWTD functions to general PWTD functions whose regular subdifferentials coincide with their basic subdifferentials.
\begin{proposition}\label{ssderive-PTD}
 For a PWTD function $\psi\!:\mathbb{R}^m\to\overline{\mathbb{R}}$, the following assertions hold. 
 \begin{enumerate}
  \item [(i)] Consider any $y\in{\rm dom}\,\psi$. For any $w\in\mathbb{R}^m$, it holds that
		\begin{equation*}
	d^2\psi(y)(w)=\lim_{\tau\downarrow 0}\Delta_{\tau}^2\psi(y)(w)
	  =\left\{\begin{array}{cl}
			\!\!\langle w,\nabla^2 \psi_{k}(y)w\rangle\ {\rm for\ each}\ k\in J_{y,w}&{\rm if}\ J_{y,w}\ne\emptyset,\\
				\infty & {\rm if}\ J_{y,w}=\emptyset,
			\end{array}\right.
		\end{equation*}
		and consequently $d^2\psi(y)$ is a PWLQ function.
		
  \item[(ii)] Fix any $y\in{\rm dom}\,\psi$ with $\widehat{\partial}\psi(y)=\partial \psi(y)$ and $v\in\partial \psi(y)$. For each $w\in\mathbb{R}^m$, 
		\begin{align*}
		d^2 \psi(y|v)(w)
		&=\left\{\begin{array}{cl}
		\!\!\langle w,\nabla^2 \psi_k(y)w\rangle\ {\rm for\ any}\ k\in J_{y,w}
				&{\rm if}\ w\in\mathcal{C}_{\psi}(y,v),\\
				\infty & {\rm if}\ w\notin\mathcal{C}_{\psi}(y,v)
			\end{array}\right.  \nonumber \\
			&=d^2 \psi(y)(w)+\delta_{\mathcal{C}_{\psi}(y,v)}(w)=\lim_{\tau\downarrow 0}\Delta_{\tau}^2\psi(y|v)(w),
	\end{align*}
	so $\psi$ is properly twice epi-differentiable at $y$ for $v$ with a proper PWLQ $d^2\psi(y|v)$.
 \end{enumerate}
\end{proposition}
\begin{proof}
 {\bf(i)} Fix any $w\in\mathbb{R}^m$. It suffices to prove the second equality. We first consider that $J_{y,w}\ne\emptyset$. Pick any $k\in J_{y,w}$. Clearly, $w\in\mathcal{T}_{\Omega_k}(y)$. By Definition \ref{ssderiv-def} and Lemma \ref{sderive-PTD} (ii), there exist $\tau_{\nu}\downarrow 0$ and $w^{\nu}\to w$ with $w^{\nu}\in\tau_{\nu}^{-1}[\,\bigcup_{i\in J_{y}}\Omega_i\!-\!y]$  such that 
 \[
  d^2\psi(y)(w)=\lim_{\nu\to\infty}\frac{\psi(y+\tau_{\!\nu}w^{\nu})-\psi(y)-\tau_{\!\nu}d\psi(y)(w^{\nu})}{\tau_{\!\nu}^2/2}.
 \]
 Obviously, $y+\tau_{\!\nu}w^{\nu}\in{\rm dom}\,\psi$ for each $\nu\in\mathbb{N}$ and $w^{\nu}\in{\rm dom}\,d\psi(y)$ for all sufficiently large $\nu\in\mathbb{N}$. Then, there exist an index $\overline{j}\in J_{y}$ and an infinite index set $N\subset\mathbb{N}$ such that for all $\nu\in N$, $y+\tau_{\nu}w^{\nu}\in\Omega_{\overline{j}}$ and $w^{\nu}\in\mathcal{T}_{\Omega_{\overline{j}}}(y)$, so that $\overline{j}\in J_{y,w}$ and
 \begin{equation}\label{d2h-equa0}
  d^2\psi(y)(w)=\lim_{\nu\xrightarrow[N]{}\infty}\frac{\psi_{\overline{j}}(y+\tau_{\!\nu}w^{\nu})-\psi_{\overline{j}}(y)-\tau_{\!\nu}\langle\nabla \psi_{\overline{j}}(y),w^{\nu}\rangle}{\tau_{\!\nu}^2/2}=\langle w,\nabla^2\psi_{\overline{j}}(y)w\rangle.
 \end{equation}
 On the other hand, by the polyhedrality of each $\Omega_{i}$ for $i\in[s]$ and \cite[Exercise 6.47]{RW98}, for any $\tau\!>\!0$ small enough, $y+\tau w\!\in\!\bigcap_{i\in J_{y,w}}\Omega_i$. Hence, for each $i\!\in\! J_{y,w}$, it holds that
 \begin{align}\label{d2h-equa1}
 \langle w,\nabla^2\psi_i(y)w\rangle
  &=\lim_{\tau\downarrow 0}\frac{\psi_i(y+\tau w)-\psi_i(y)-\tau\langle\nabla \psi_i(y),w\rangle}{\tau^2/2} \nonumber\\
  &=\lim_{\tau\downarrow 0}\frac{\psi(y+\tau w)-\psi(y)-\tau d\psi(y)(w)}{\tau^2/2}.
 \end{align} 
 Recall that $k\in J_{y,w}$. Together with the above equations \eqref{d2h-equa0} and \eqref{d2h-equa1}, 
 \[
  d^2\psi(y)(w)=\langle w,\nabla^2\psi_{\overline{j}}(y)w\rangle=\lim_{\tau\downarrow 0}\Delta_{\tau}^2\psi(y)(w)=\langle w,\nabla^2 \psi_k(y)w\rangle.
 \]
 By the arbitrariness of $k\in J_{y,w}$, the conclusion holds for the case $J_{y,w}\ne\emptyset$. When $J_{y,w}=\emptyset$, as $y\in{\rm dom}\,\psi$, we deduce from Lemma \ref{sderive-PTD} (ii) that $w\notin\mathcal{T}_{{\rm dom}\,\psi}(y)$. Consequently, for any sufficiently small $\tau>0$ and any $w'$ sufficiently close to $w$, $y+\tau w'\notin{\rm dom}\,\psi$, which implies that $d^2\psi(y)(w)=\infty=\lim_{\tau\downarrow 0}\Delta_{\tau}^2\psi(y)(w)$. Thus, the conclusion of part (i) holds.  

 \noindent
 {\bf(ii)} Fix any $w\in\mathbb{R}^m$. We first consider that $w\in\mathcal{T}_{{\rm dom}\,\psi}(y)$. By Definition \ref{ssderiv-def} and Lemma \ref{sderive-PTD} (ii), there exist $\tau_{\nu}\downarrow 0$ and $w^{\nu}\to w$ with $w^{\nu}\in\tau_{\nu}^{-1}[\,\bigcup_{j\in J_{y}}\Omega_j-y]$ such that 
 \[
  d^2\psi(y|v)(w)=\lim_{\nu\to\infty}\frac{\psi(y+\tau_{\nu}w^{\nu})-\psi(y)-\tau_{\!\nu}\langle v,w^{\nu}\rangle}{\tau_{\nu}^2/2}.
 \]
 Obviously, for each $\nu\in\mathbb{N}$, $y+\tau_{\!\nu}w^{\nu}\in{\rm dom}\,\psi$. Then, there exist an index $\widetilde{j}\in J_{y}$ and an infinite index set $N\subset\mathbb{N}$ such that for all $\nu\in N$, $y+\tau_{\nu}w^{\nu}\in\Omega_{\widetilde{j}}$ and $w^{\nu}\in\mathcal{T}_{\Omega_{\widetilde{j}}}(y)$, and consequently, $\widetilde{j}\in\!J_{y,w}$ and
 \begin{align}\label{d2h-equa2}
  d^2\psi(y|v)(w)&=\lim_{\nu\xrightarrow[N]{}\infty}\frac{\psi_{\widetilde{j}}(y+\tau_{\!\nu}w^{\nu})-\psi_{\widetilde{j}}(y)-\tau_{\!\nu}\langle v,w^{\nu}\rangle}{\tau_{\nu}^2/2}\nonumber\\
  &=\lim_{\nu\xrightarrow[N]{}\infty}\Big[\!\langle w^{\nu},\nabla^2\psi_{\widetilde{j}}(y)w^{\nu}\rangle+\frac{2(d\psi(y)(w^{\nu})-\langle v,w^{\nu}\rangle)}{\tau_{\!\nu}}+\frac{o(\tau_{\!\nu}^2)}{\tau_{\!\nu}^2}\Big]\nonumber\\
  &\ge\left\{\begin{array}{cl}
		\langle w,\nabla^2\psi_{\widetilde{j}}(y)w\rangle &{\rm if}\ w\in\mathcal{C}_{\psi}(y,v),\\
		\infty&{\rm if}\ w\in\mathcal{T}_{{\rm dom}\,\psi}(y)\backslash\mathcal{C}_{\psi}(y,v),
  \end{array}\right.	 
 \end{align}
 where the inequality for the case $w\in\mathcal{C}_{\psi}(y,v)$ is due to $v\in\widehat{\partial}\psi(y)$ and the equivalence in \eqref{Rsdiff-sderiv}, and the inequality for the case $w\in\mathcal{T}_{{\rm dom}\,\psi}(y)\backslash\mathcal{C}_{\psi}(y,v)$ is obtained by using the continuity of $d\psi(y)$ relative to its domain by Proposition \ref{sderive-PTD} (i). Note that $\widetilde{j}\in\!J_{y,w}$. Combining the above inequality \eqref{d2h-equa2} and equality \eqref{d2h-equa1} yields that
 \[
  d^2\psi(y|v)(w)=\left\{\begin{array}{cl}
	\!\langle w,\nabla^2 \psi_k(y)w\rangle\ {\rm for\ each}\ k\in J_{y,w}
	&{\rm if}\ w\in\mathcal{C}_{\psi}(y,v),\\
	\infty & {\rm if}\ w\in\mathcal{T}_{{\rm dom}\,\psi}(y)\backslash\mathcal{C}_{\psi}(y,v).
 \end{array}\right. 
 \]
 When $w\notin\mathcal{T}_{{\rm dom}\,\psi}(y)$, since $\mathcal{T}_{{\rm dom}\,\psi}(y)=\big\{w\in\mathbb{R}^m\,|\,\liminf_{\tau\downarrow 0}{\rm dist_2}\big(w,\frac{{\rm dom}\,\psi-y}{\tau}\big)=0\big\}$, there exists $\varepsilon>0$ such that for any $\tau\!\in(0,\varepsilon)$ and any $w'$ with $\|w'\!-w\|_2\le\varepsilon$, $w'\notin\tau^{-1}[{\rm dom}\,\psi-y]$. By Definition \ref{ssderiv-def}, $d^2\psi(y|v)(w)=\infty=d^2\psi(y)(w)$. Together with the above equation, we obtain the desired result. 
 \end{proof}

 In Proposition \ref{ssderive-PTD} (ii), the restriction $\widehat{\partial}\psi(y)=\partial \psi(y)$ on $y\in{\rm dom}\,\psi$ is crucial to the properness of $d^2\psi(y|v)$. For example, for the piecewise linear $h(t):=\max\{0,\min\{1,t\}\}$ and $\overline{t}=1$, an elementary calculation yields that $\emptyset=\widehat{\partial} h(\overline{t})\ne\partial h(\overline{t})=\{0,1\}$, and $dh(\overline{t})(w)=\left\{\begin{array}{cl}
 w &{\rm if}\ w\leq 0,\\
 0 &{\rm if}\ w>0.
 \end{array}\right.$ Pick $\overline{v}=1\in\partial h(\overline{t})$. As $\overline{v}w> dh(\overline{t})(w)=0$ for all $w>0$, we have $d^2 h(\overline{t}|\overline{v})(w)=-\infty$ by \cite[Proposition 13.5]{RW98}, so $d^2h(\overline{t}|\overline{v})$ is not proper.

 Next we characterize the parabolic subderivatives of PWTD functions, and show that a PWTD function $\psi$ is parabolically epi-differentiable at any $y\in{\rm dom}\,\psi$ for each $w\in{\rm dom}\,d\psi(y)$ under the same condition as in Proposition \ref{ssderive-PTD}.
 \begin{proposition}\label{psderive-PTD}
 Let $\psi\!:\mathbb{R}^m\to\overline{\mathbb{R}}$ be a PWTD function. Consider any $y\in{\rm dom}\,\psi$ and $w\in{\rm dom}\,d\psi(y)$. The following assertions hold.
 \begin{enumerate}
 \item[(i)]  $\mathcal{T}^2_{{\rm dom}\,\psi}(y,w)=\bigcup_{k\in J_{y,w}}\!\mathcal{T}^2_{\Omega_k}(y,w)=\mathcal{T}_{\mathcal{T}_{{\rm dom}\,\psi}(y)}(w)$.
		
 \item [(ii)]  $d^2\psi(y)(w|z)<\infty$ if and only if $z\in\mathcal{T}^2_{{\rm dom}\,\psi}(y,w)$.
		
 \item[(iii)] For any $z\in\mathbb{R}^m$, it holds that
		\begin{equation}\label{d2h-equa} 
		-\infty<d^2\psi(y)(w|z)=\lim_{\tau\downarrow 0}\frac{\psi\big(y+\tau w+\frac{1}{2}\tau^2 z\big)-\psi(y)-\tau d\psi(y)(w)}{\tau^2/2}, 
		\end{equation} 
	and consequently $\psi$ is parabolically epi-differentiable at $y$ for $w$.
		
 \item[(iv)] For any $z\in\mathbb{R}^m$, $d^2\psi(y)(w|z)\!=\!d^2\psi(y)(w)+d(d\psi(y)(w))(z)$. If $\psi$ is regular at $y$, i.e. its epigraph ${\rm epi}\,\psi$ is Clarke regular at $(y,\psi(y))$, then it holds that
		\begin{equation*}
		d^2 \psi(y)(w|z)=d^2 \psi(y)(w)+\sup_{\xi\in\mathcal{A}_{\psi}(y,w)}\langle\xi,z\rangle\quad\ \forall z\in\mathbb{R}^m,
		\end{equation*}
		where $\mathcal{A}_{\psi}(y,w)\!:=\!\big\{\xi\in\partial \psi(y)\ |\ \langle \xi,w\rangle=d\psi(y)(w)\}$.
	\end{enumerate}
 \end{proposition}
 \begin{proof} 
  {\bf(i)} From Lemma \ref{sderive-PTD} (ii), $w\in \mathcal{T}_{{\rm dom}\,\psi}(y)$; while from \cite[Proposition 3.37]{BS00}, $\mathcal{T}^2_{{\rm dom}\,\psi}(y,w)=\bigcup_{i=1}^{s}\!\mathcal{T}^2_{\Omega_i}(y,w)$. For each $\Omega_i$ with $i\in[s]$, invoking \cite[Proposition 13.12]{RW98} leads to $\mathcal{T}^2_{\Omega_i}(y,w)=\mathcal{T}_{\mathcal{T}_{\Omega_i}(y)}(w)$. Thus, for each $k\in[s]$, $z\in\mathcal{T}^2_{\Omega_k}(y,w)$ implies that $y\in\Omega_k$ and $w\in\mathcal{T}_{\Omega_k}(y)$, i.e., $k\in J_{y,w}$. Consequently,   $\bigcup_{i=1}^{s}\mathcal{T}^2_{\Omega_i}(y,w)=\bigcup_{k\in J_{y,w}}\!\mathcal{T}^2_{\Omega_k}(y,w)$. Together with \cite[Proposition 3.37]{BS00}, it follows that
  \[
  \mathcal{T}^2_{{\rm dom}\,\psi}(y,w)=\bigcup_{i=1}^{s}\mathcal{T}^2_{\Omega_i}(y,w)=\bigcup_{k\in J_{y,w}}\!\mathcal{T}^2_{\Omega_k}(y,w)=\mathcal{T}_{\bigcup_{k\in J_{y,w}}\!\mathcal{T}_{\Omega_k}(y)}(w)=\mathcal{T}_{\mathcal{T}_{{\rm dom}\,\psi}(y)}(w).
  \]

  \noindent
  {\bf(ii)} Pick any $z\in\mathbb{R}^m$ with $d^2\psi(y)(w|z)<\infty$. By Definition \ref{psderiv-def}, there exist sequences $\tau_{\nu}\downarrow 0$ and $z^{\nu}\to z$ such that
  \begin{equation*}
  \infty>d^2\psi(y)(w|z)=\lim_{\nu\to \infty}\frac{\psi(y+\tau_{\!\nu} w+\frac{1}{2}\tau_{\!\nu}^2 z^{\nu})-\psi(y)-\tau_{\!\nu}d\psi(y)(w)}{\tau_{\nu}^2/2},
  \end{equation*}
  which implies that $y+\tau_{\!\nu}w+\frac{1}{2}\tau_{\!\nu}^2 z^{\nu}\in {\rm dom}\,\psi$ for all sufficiently large $\nu$. Then, there exists an index $\overline{i}\!\in\![s]$ and an infinite index set $N\subset\mathbb{N}$ such that $\{y+\tau_{\!\nu}w+\frac{1}{2}\tau_{\!\nu}^2 z^{\nu}\}_{\nu\in N}\!\subset\! \Omega_{\overline{i}}$. Hence, $w\in\mathcal{T}_{\Omega_{\overline{i}}}(y)$ and $z\in\mathcal{T}^2_{\Omega_{\overline{i}}}(y,w)\subset\mathcal{T}^2_{{\rm dom}\,\psi}(y,w)$. Moreover, from the last equation, 
  \begin{align}\label{pd2h-equa0}
   d^2\psi(y)(w|z)
   &=\lim_{\nu\xrightarrow[N]{} \infty}\frac{\psi_{\overline{i}}(y+\tau_{\!\nu} w+\frac{1}{2}\tau_{\!\nu}^2 z^{\nu})-\psi_{\overline{i}}(y)-\tau_{\!\nu}\langle \nabla \psi_{\overline{i}}(y), w\rangle}{\tau_{\nu}^2/2}\nonumber\\
   &=\lim_{\nu\xrightarrow[N]{}\infty}\frac{\tau_{\!\nu}^2\langle\nabla \psi_{\overline{i}}(y), z^\nu\rangle+\langle \tau_{\nu}w+\frac{1}{2}\tau_{\nu}^2z^\nu,\nabla ^2 \psi_{\overline{i}}(y)(\tau_{\nu}w+\frac{1}{2}\tau_{\nu}^2z^\nu)\rangle+o(\tau_{\nu}^2)}{\tau_{\nu}^2}\nonumber\\
   &=\langle\nabla \psi_{\overline{i}}(y),z\rangle+\langle w,\nabla^2\psi_{\overline{i}}(y)w\rangle.
  \end{align}
  Conversely, pick any $z\!\in\!\mathcal{T}^2_{{\rm dom}\,\psi}(y,w)$. Then, there exists $\widetilde{i}\!\in\![s]$ such that $z\in\mathcal{T}^2_{\Omega_{\widetilde{i}}}(y,w)$. By the proof of part (i), $\widetilde{i}\!\in\! J_{y,w}$. Recall that $\Omega_{\widetilde{i}}$ is polyhedral. For any sufficiently small $\tau\!>\!0$, $y+\tau w+\frac{1}{2}\tau^2 z\in\Omega_{\widetilde{i}}$. Consequently, it holds that
  \begin{align*}
  d^2\psi(y)(w|z)
  &\le\lim_{\tau\downarrow 0}\frac{\psi(y+\tau w+\frac{1}{2}\tau^2z)-\psi(y)\!-\!\tau d\psi(y)(w)}{\frac{1}{2}\tau^2}\\
  &=\lim_{\tau\downarrow 0}\frac{\psi_{\widetilde{i}}(y+\tau w+\frac{1}{2}\tau^2z)
 		-\psi_{\widetilde{i}}(y)\!-\!\tau d\psi_{\widetilde{i}}(y)(w)}{\frac{1}{2}\tau^2}\\
   &=\langle w,\nabla^2 \psi_{\widetilde{i}}(y)w\rangle+\langle\nabla \psi_{\widetilde{i}}(y), z\rangle.
  \end{align*}
  This shows that $d^2\psi(y)(w|z)<\infty$, and the desired equivalence then follows. 

  \noindent
  {\bf(iii)} Fix any $z\!\in\!\mathbb{R}^m$. We proceed the arguments by two cases: $z\!\notin\!\mathcal{T}^2_{{\rm dom}\,\psi}(y,w)$ and $z\in\mathcal{T}^2_{{\rm dom}\,\psi}(y,w)$. If $z\notin\mathcal{T}^2_{{\rm dom}\,\psi}(y,w)$, for any $\tau>0$ small enough, $y+\tau w+\frac{1}{2}\tau^2z\notin{\rm dom}\,\psi$. Together  with $w\in{\rm dom}\,d\psi(y)$, it follows that
  \[
   d^2\psi(y)(w|z)=\infty=\lim_{\tau\downarrow 0}\frac{\psi(y+\tau w+\frac{1}{2}\tau^2 z)-\psi(y)-\tau d\psi(y)(w)}{\tau^2/2}.
  \]
  That is, equation \eqref{d2h-equa} holds for this case. Next we consider that $z\in\mathcal{T}^2_{{\rm dom}\,\psi}(y,w)=\bigcup_{k\in J_{y,w}}\!\mathcal{T}^2_{\Omega_k}(y,w)$. For convenience, write $J_{y,w,z}\!:=\{k\!\in \!J_{y,w}\,|\,z\!\in\!\mathcal{T}^2_{\Omega_k}(y,w)\}$. For each $i\!\in\!J_{y,w,z}$, from $z\in\mathcal{T}^2_{\Omega_i}(y,w)$, the polyhedrality of  $\Omega_{i}$ and \cite[Proposition 13.12]{RW98}, for any sufficiently small $\tau>0$, $y+\tau w+\frac{1}{2}\tau^2z\in\Omega_i$. Thus, for each $i\in J_{y,w,z}$, by using Lemma \ref{sderive-PTD} (ii) and an elementary calculation, it holds that
  \begin{align}\label{pd2h-equa1}
   \langle w,\nabla^2 \psi_{i}(y)w\rangle+\langle\nabla \psi_{i}(y), z\rangle
   &=\lim_{\tau\downarrow 0}\frac{\psi_i(y+\tau w+\frac{1}{2}\tau^2z)-\psi_i(y)-\tau \langle \psi_i(y),w\rangle}{\tau^2/2}\nonumber\\
   &=\lim_{\tau\downarrow 0}\frac{\psi(y+\tau w+\frac{1}{2}\tau^2z)-\psi(y)-\tau d\psi(y)(w)}{\tau^2/2}.
  \end{align}
  In addition, from the proof of the necessity for part (ii), there exists an index $\overline{i}\in J_{y,w,z}$. By combining \eqref{pd2h-equa1} and \eqref{pd2h-equa0}, we obtain \eqref{d2h-equa}. Note that $0\in\mathcal{T}^2_{\Omega_k}(y,w)$ for each $k\in J_{y,w}$. From part (ii), we have $d^2\psi(y)(w|0)<\infty$. This, along with \eqref{d2h-equa}, shows that $\psi$ is parabolically epi-differentiable at $y$ for $w$.  

  \noindent
  {\bf(iv)} Fix any $z\in\mathbb{R}^m$. We first consider that $J_{y,w,z}\ne\emptyset$. Pick any $i\in\!J_{y,w,z}$. By the polyhedrality of  $\Omega_{i}$, for any $\tau>0$ small enough, $y+\tau w+\frac{1}{2}\tau^2z\in\Omega_i$. From part (iii), 
  \begin{align*}
  d^2\psi(y)(w|z)
  &=\lim_{\tau\downarrow 0}
 	\frac{\psi_i(y+\tau w+\frac{1}{2}\tau^2z )-\psi_i(y)-\tau d\psi_i(y)(w)}{\tau^2/2}\nonumber\\
  &=\langle w,\nabla^2\psi_i(y)w\rangle+\langle\nabla \psi_i(y),z\rangle
 	=d^2\psi(y)(w)+\langle\nabla \psi_i(y),z\rangle, 
  \end{align*}
  where the second equality is using $d\psi(y)(w)=\langle\nabla \psi_i(y),w\rangle$ and the third one is due to Proposition \ref{ssderive-PTD} (i). When $J_{y,w,z}=\emptyset$, by the definition of $J_{y,w,z}$, for each $k\in J_{y,w}$, $z\notin\mathcal{T}_{\Omega_k}^2(y,w)$. From the proof of part (i), $z\notin\mathcal{T}_{\mathcal{T}_{{\rm dom}\,\psi}(y)}(w)={\rm dom}\,d^2\psi(y)(w|\cdot)$, and consequently, $d^2\psi(y)(w|z)=\infty$. The above discussions show that  
  \begin{equation}\label{pd2h-equa3}
   d^2\psi(y)(w|z)=d^2\psi(y)(w)+\left\{\begin{array}{cl}
 		\!\langle\nabla \psi_k(y),z\rangle\ {\rm for\ any}\ k\in J_{y,w,z}&{\rm if}\ J_{y,w,z}\ne\emptyset,\\
 		\infty&{\rm if}\ J_{y,w,z}=\emptyset. 
 	\end{array}\right.
  \end{equation}
  Let $\psi_{y}=d\psi(y)$. By Lemma \ref{sderive-PTD} (ii)-(iii), $\psi_y$ is a proper piecewise linear function with ${\rm dom}\,\psi_{y}=\bigcup_{k\in J_{y}}\!\mathcal{T}_{\Omega_k}(y)$. For any $w\in\mathbb{R}^m$ and $z\!\in\!\mathbb{R}^m$, define the index sets
  \[
   I_{w}\!:=\!\{i\!\in\! J_{y}\,|\,w\in\mathcal{T}_{\Omega_i}(y)\}\ \ {\rm and}\ \ I_{w,z}\!:=\!\{i\in I_{w}\,|\,z\!\in\!\mathcal{T}_{\mathcal{T}_{\Omega_i}(y)}(w)\}. 
  \]
  Applying the conclusion of Lemma \ref{sderive-PTD} (iii) to the function $\psi_y$ yields that 
  \[
   d\psi_{y}(w)(z)=\left\{\begin{array}{cl}
 	\!\langle \nabla \psi_k(y),z\rangle\ {\rm for\ any}\ k\in I_{w,z}&{\rm if}\ I_{w,z}\ne\emptyset,\\
 	\infty&{\rm if}\ I_{w,z}=\emptyset. 
   \end{array}\right.
  \]
  Note that $I_{w,z}\!=\!J_{y,w,z}$ because $\mathcal{T}^2_{\Omega_i}(y,w)\!=\!\mathcal{T}_{\mathcal{T}_{\Omega_i}(y)}(w)$ for each $i\in[s]$. Combining the above equality with \eqref{pd2h-equa3} yields the first part of the conclusions. Now assume that $\psi$ is regular at $y$. Invoking \cite[Theorem 8.30]{RW98} for $\psi$ and the properness of $d\psi(y)$ by Lemma  \ref{sderive-PTD} (iii) leads to $\psi_y(w')=\sup_{v\in\partial \psi(y)}\langle v,w'\rangle$ for any $w'\!\in\!\mathbb{R}^m$, which implies that $\psi_y$ is  lsc convex. Along with its properness, $\partial \psi_y(w)=\mathop{\arg\max}_{v\in\partial \psi(y)}\langle v,w\rangle$, and consequently
  \[
   d\psi_y(w)(z)=\sup_{\xi\in\partial \psi_y(w)}\langle \xi,z\rangle
   =\sup_{\xi\in\mathbb{R}^m}\Big\{\langle \xi,z\rangle\ \ {\rm s.t.}\ \ 
   \xi\in\partial \psi(y),\,\langle\xi,w\rangle=d\psi(y)(w)\Big\},
  \] 
  where the first equality is obtained by using \cite[Theorem 8.30]{RW98} for $\psi_y$ and the properness of $d\psi_{y}(w)$ due to Lemma \ref{sderive-PTD} (iii). Combining above equalities and the first part of the conclusions and using the definition of $\mathcal{A}_{\psi}(y,w)$ yields the result. 
 \end{proof}

 Proposition \ref{psderive-PTD} extends the result of \cite[Exercise 13.61]{RW98} for PWLQ convex functions to PWTD functions. By Propositions \ref{psderive-PTD} (iv) and \ref{ssderive-PTD} (ii), we can prove the parabolic regularity of PWTD functions, which partly extends the result of \cite[Theorem 13.67]{RW98}.
\begin{proposition}
 Let $\psi\!:\mathbb{R}^m\to\overline{\mathbb{R}}$ be a PWTD function. Consider any $y\in{\rm dom}\,\psi$ and $w\in{\rm dom}\,d\psi(y)$. Let $\varphi(z)\!:=d^2\psi(y)(w|z)$ for $z\in\mathbb{R}^m$. The following assertions hold.
 \begin{enumerate}
 \item [(i)] $\varphi$ is a proper lsc function with $\varphi^*=-d^2\psi(y)(w)+h^*$, where $h(\cdot)=d(d\psi(y)(w))(\cdot)$. 
		
 \item[(ii)] Suppose that $\psi$ is regular at $y$. Then, $\varphi$ is a proper lsc convex function with 
 \[
  \varphi^*(z^*)=-d^2\psi(y)(w)+\delta_{\mathcal{A}_{\psi}(y,w)}(z^*)\quad{\rm for}\ z^*\in\mathbb{R}^m, 
  \]
  and the function $\psi$ is parabolically regular at $y$ for every $v\in\partial \psi(y)$.   
 \end{enumerate}
\end{proposition}
 \begin{proof}  
 {\bf(i)} By the first part of Proposition \ref{psderive-PTD} (iv), $\varphi(z)=d^2\psi(y)(w)+d(d\psi(y)(w))(z)$ for $z\in\mathbb{R}^m$. The expression of $\varphi^*$ then follows the definition of conjugate functions. 

 \noindent
 {\bf(ii)} Since $\psi$ is regular at $y$, from the second part of Proposition \ref{psderive-PTD} (iv), $\varphi$ is a sum of $d^2\psi(y)(w)$ and the support function of the closed convex set $\mathcal{A}_{\psi}(y,w)$, so is a proper lsc convex function. The expression of $\varphi^*$ follows the definition of conjugate functions. Fix any $v\in\partial \psi(y)$. Pick any $w\in\mathcal{C}_{\psi}(y,v)$. From the expression of $\varphi^*$, it follows that
 \[
  -\varphi^*(v)=d^2\psi(y)(w)-\delta_{\mathcal{A}_{\psi}(y,w)}(v)\le d^2\psi(y)(w)=d^2\psi(y|v)(w),
 \] 
 where the second equality is due to Proposition \ref{ssderive-PTD} (ii) and $w\in\mathcal{C}_{\psi}(y,v)$. By the definition of conjugate functions, $-\varphi^*(v)\!=\!\inf_{z\in\mathbb{R}^m}\big\{d^2\psi(y)(w|z)-\langle v,z\rangle\}\!\ge\! d^2\psi(y|v)(w)$, where the inequality is due to \cite[Proposition 13.64]{RW98}. Together with the above equation,  $\inf_{z\in\mathbb{R}^m}\big\{d^2\psi(y)(w|z)-\langle v,z\rangle\}\!=\!d^2\psi(y|v)(w)$. This, by Definition \ref{pregular-def}, shows that $\psi$ is parabolically regular at $y$ for $v$. The result follows the arbitrariness of $v\in\partial \psi(y)$.
 \end{proof}

 \section{Twice epi-differentiability of $f$}\label{sec4}
 
 Before establishing the twice epi-differentiability of $f$, we first prove its parabolic epi-differentiability. This requires the following technical lemma. Since its proof can be found in \cite[Theorem 3.4]{Moh22-archive} by Assumption \ref{ass0} and Lemma \ref{sderive-PTD}, here we do not include it.
 \begin{lemma}\label{fsubderive-lemma}
  Fix any $\overline{x}\in{\rm dom}\,f$. If the MSQC holds for system $F(x)\in{\rm dom}\,\vartheta$ at $\overline{x}$, then $df(\overline{x})(w)=d\vartheta(F(\overline{x}))(dF(\overline{x})(w))$ for $w\in\mathbb{X}$, and consequently, $df(\overline{x})$ is a proper and piecewise positively homogeneous continuous function (i.e., ${\rm dom}\,df(\overline{x})$ is nonempty and can be represented as a union of finitely many polyhedral sets, say $\bigcup_{i=1}^lD_i$ for polyhedral sets $D_1,\ldots,D_l$, and for each $i\in[l]$, there is a function $h_i$, which is a positively homogeneous continuous function on a superset of $D_i$, such that $df(\overline{x})=h_i$ on $D_i$).
 \end{lemma}

 By Lemma \ref{fsubderive-lemma}, the critical cone of the function $f$ can be characterized as follows. 
 \begin{lemma}\label{ccone-ffun}
  Consider any $\overline{x}\in{\rm dom}f$ and $\overline{v}\in\partial\!f(\overline{x})$. Suppose that the MSQC holds for constraint system $F(x)\in{\rm dom}\,\vartheta$ at $\overline{x}$. Then, the following assertions hold.
  \begin{enumerate}
  \item [(i)] $w\in\mathcal{C}_{\!f}(\overline{x},\overline{v})$ iff
 		$d\vartheta(F(\overline{x}))(dF(\overline{x})(w))
 		=\langle\overline{v},w\rangle$, and in this case
 		$w\in\mathcal{T}_{{\rm dom}\,f}(\overline{x})$.
 		
 \item[(ii)] For any $w\in \mathcal{T}_{{\rm dom}\,f}(\overline{x})$,
 $\mathcal{T}^2_{{\rm epi}\,f}\big((\overline{x},F(\overline{x})),(w,dF(\overline{x})(w))\big)\ne\emptyset$.
 		
  \item [(iii)] $\mathcal{C}_{\!f}(\overline{x},\overline{v})\subset
 		{\rm dom}\,d^2\!f(\overline{x}|\overline{v})$, and the converse inclusion also holds if $\overline{v}\in\widehat{\partial}f(\overline{x})$.
 		
  \item[(iv)] $\mathcal{C}_{\!f}(\overline{x},\overline{v})\ne\emptyset$ whenever $dF(\overline{x})(0)=0$.
 \end{enumerate}
 \end{lemma}
 \begin{proof}  
 {\bf (i)} The equivalence follows the definition of $\mathcal{C}_{\!f}(\overline{x},\overline{v})$ and Lemma \ref{fsubderive-lemma}. For any $w$ with $d\vartheta(F(\overline{x}))(dF(\overline{x})(w))=\langle \overline{v},w\rangle$, we have $dF(\overline{x})(w)\in {\rm dom}\,d\vartheta(F(\overline{x}))=\mathcal{T}_{{\rm dom}\,\vartheta}(F(\overline{x}))$, where the equality is due to Lemma \ref{sderive-PTD} (ii) with $\psi=\vartheta$. By Lemma \ref{tcone-domf}, $w\in\mathcal{T}_{{\rm dom}f}(\overline{x})$.
 
 \noindent
 {\bf(ii)} Fix any $w\in\!\mathcal{T}_{{\rm dom}f}(\overline{x})$. By Corollary \ref{Stcone-domf},  $\mathcal{T}^2_{{\rm dom}\,f}(\overline{x},w)\ne\emptyset$. Pick any $z\in\!\mathcal{T}^2_{{\rm dom}\,f}(\overline{x},w)$. Then, there exist $\tau_k\downarrow 0$ and $\{x^k\}_{k\in\mathbb{N}}\subset{\rm dom}\,f$ with $x^k=\overline{x}+\tau_k w+\frac{1}{2}\tau^2_kz+o(\tau^2_k)$. From the parabolic semidifferentiability of $F$, for all sufficiently large $k$, it holds that
 \[
 \bigcup_{i=1}^sC_i={\rm dom}\,\vartheta\ni F(x^k)=F(\overline{x})+\tau_kdF(\overline{x})(w)
 +\frac{1}{2}\tau^2_k F''(\overline{x};w,z)+o(\tau^2_k).
 \]
 Clearly, there exist $i\in [s]$ and an infinite index set $N\!\subset\mathbb{N}$ such that
 for all $k\in N$,
 \[
  C_i\ni F(x^k)=F(\overline{x})+\tau_kdF(\overline{x})(w)
 +\frac{1}{2}\tau^2_k F''(\overline{x};w,z)+o(\tau^2_k).
 \]
 Note that $dF(\overline{x})(w)\in \mathcal{T}_{C_{i}}(F(\overline{x}))$. From Lemma \ref{fsubderive-lemma} and Lemma \ref{sderive-PTD} (ii), we have
 \begin{equation}\label{temp-dfequa1}
 	df(\overline{x})(w)=d\vartheta(F(\overline{x}))(dF(\overline{x})(w))
 	=\langle\nabla\!\vartheta_{j}(F(\overline{x})),dF(\overline{x})(w)\rangle.
 \end{equation}
 From $F(x^k)\!\in\! C_i$, equation \eqref{temp-dfequa1} and the twice differentiability of $\vartheta_{j}$ at $F(\overline{x})$, 
 \begin{align*}
  f(x^k)&=\vartheta_{i}(F(x^k))=\vartheta_{i}\big(F(\overline{x})+\tau_kdF(\overline{x})(w)+\frac{1}{2}\tau^2_k F''(\overline{x};w,z)+o(\tau^2_k)\big)\\
 	&=\vartheta_{i}(F(\overline{x}))\!+\tau_kdF(\overline{x})(w)\!+\frac{1}{2}\tau^2_k
 	\langle\nabla\!\vartheta_{i}(F(\overline{x})),F''(\overline{x};w,z)\rangle\\
 	&\quad\!+\!\frac{1}{2}\tau^2_k\langle dF(\overline{x})(w),\nabla^2\vartheta_i(F(\overline{x}))dF(\overline{x})(w)\rangle\!+\!o(\tau^2_k).
 \end{align*}
 This, along with $(x^k,f(x^k))\in{\rm epi}\,f$ and $(w,df(\overline{x})(w))\in\mathcal{T}_{{\rm epi}\,f}(\overline{x},f(\overline{x}))$, implies that
 \begin{align*}
 &\big(z,\langle\nabla\!\vartheta_{j}(F(\overline{x})),F''(\overline{x};w,z)\rangle
 \!+\!\langle dF(\overline{x})(w),\nabla^2\vartheta_j(F(\overline{x}))dF(\overline{x})(w)\rangle\big)\\
 &\in\!\mathcal{T}^2_{{\rm epi}f}((\overline{x},f(\overline{x})),(w,df(\overline{x})w)).
 \end{align*}
 Consequently, $\mathcal{T}^2_{{\rm epi}\,f}((\overline{x},F(\overline{x})),(w,dF(\overline{x})w))$ is nonempty, and part (ii) holds.

 \noindent
 {\bf (iii)} Pick any $w\in\mathcal{C}_{f}(\overline{x},\overline{v})$. From part (i), $w\!\in\! \mathcal{T}_{{\rm dom}\,f}(\overline{x})$; while from part (ii), $\mathcal{T}^2_{{\rm epi}\,f}((\overline{x},f(\overline{x})),(w,dF(\overline{x})(w)))\ne\emptyset$. Pick any $(u,\varpi)\in\mathcal{T}^2_{{\rm epi}f}((\overline{x},f(\overline{x})),(w,dF(\overline{x})(w)))$. Then there exist $\tau_k\downarrow 0$ and $(u^k,\varpi_k)\to (u,\varpi)$ such that $(\overline{x},f(\overline{x}))+\tau_k(w,df(\overline{x})(w)) +\frac{1}{2}\tau_k^2(u^k,\varpi_k)\in {\rm epi}\,f$ for each $k\in\mathbb{N}$. Together with $df(\overline{x})(w)=\langle\overline{v},w\rangle$, for each $k\in\mathbb{N}$, 
 \begin{align*}
 \varpi_k
  &\ge\frac{f(\overline{x}+\tau_kw+\frac{1}{2}\tau_k^2 u^k)\!-f(\overline{x})
 			\!-\tau_k df(\overline{x})(w)}{\tau_k^2/2}\\
  &=\frac{f(\overline{x}+\!\tau_k(w\!+\frac{1}{2}\tau_ku^k))\!-f(\overline{x})
 			\!-\tau_k\langle\overline{v},w\!+\frac{1}{2}\tau_ku^k\rangle}{\tau_k^2/2}+\langle\overline{v},u^k\rangle.
 \end{align*}
 Passing the limit $k\to \infty$ to the last inequality yields that $\varpi\ge d^2\!f(\overline{x}|\overline{v})(w)\!+\!\langle\overline{v},u\rangle$, which means that $w\in{\rm dom}\,d^2\!f(\overline{x}|\overline{v})$ and the inclusion $\mathcal{C}_{f}(\overline{x},\overline{v})\subset {\rm dom}\,d^2\!f(\overline{x}|\overline{v})$. For the converse inclusion, as $\overline{v}\in\widehat{\partial}\!f(\overline{x})$, from \eqref{Rsdiff-sderiv} we have $df(\overline{x})(w')\ge\langle\overline{v},w'\rangle$ for all $w'\in\mathbb{X}$; while ${\rm dom}\,d^2\!f(\overline{x}|\overline{v})\subset\{w\in\mathbb{X}\,|\, df(\overline{x})(w)\le\langle \overline{v},w\rangle\}$ by \cite[Proposition 13.5]{RW98}. Thus, ${\rm dom}\,d^2\!f(\overline{x}|\overline{v})\subset\{w\in\mathbb{X}\,|\, df(\overline{x})(w)=\langle\overline{v},w\rangle\}=\mathcal{C}_{f}(\overline{x},\overline{v})$.

 \noindent
 {\bf(iv)} When $dF(\overline{x})(0)=0$, we have $d\vartheta(F(\overline{x}))(dF(\overline{x})(0))=0$ by Lemma \ref{sderive-PTD} (iii), which along with part (i) implies that $0\in\mathcal{C}_{f}(\overline{x},\overline{v})$. Consequently,  $\mathcal{C}_{f}(\overline{x},\overline{v})\ne\emptyset$. 
 \end{proof}
 
 Lemma \ref{ccone-ffun} (iii) extends the conclusion of \cite[Theorem 4.4]{MohMS22-MOR} for the fully subamenable function to a large class of nonconvex composite functions of form \eqref{ffun}. Now we are ready to characterize the parabolic subderivative of $f$ via Proposition \ref{prop-fpsubderive}, which extends the conclusion of \cite[Theorem 4.4]{MohS20-SIOPT} to the composite function of form \eqref{ffun}. 
 \begin{proposition}\label{prop-fpsubderive}
  Let $\overline{x}\in{\rm dom}\,f$. Suppose that the MSQC holds for constraint system $F(x)\in{\rm dom}\,\vartheta$ at $\overline{x}$. Fix any $w\in\mathcal{T}_{{\rm dom}\,f}(\overline{x})$. Then, the following assertions hold.
 \begin{enumerate}
  \item [(i)] For any $z\in \mathbb{X}$, $-\infty<d^2\!f(\overline{x})(w|z)		=d^2\vartheta(F(\overline{x}))(dF(\overline{x})(w)\,|\,F''(\overline{x};w,z))$ with
 		\begin{equation}\label{fdomain-psderiv}
 		{\rm dom}\,d^2f(\overline{x})(w\,|\,\cdot)=\mathcal{T}^2_{{\rm dom}\,f}(\overline{x},w).
 		\end{equation}
 		
 \item [(ii)] The function $f$ is parabolically epi-differentiable at $\overline{x}$ for $w$.
 \end{enumerate}
 \end{proposition}
 \begin{proof}  
 {\bf(i)} Since $w\in\mathcal{T}_{{\rm dom}\,f}(\overline{x})$, from Lemma \ref{tcone-domf}, $dF(\overline{x})(w)\in\mathcal{T}_{{\rm dom}\,\vartheta}(F(\overline{x}))$, which by Lemma \ref{sderive-PTD} (ii) for $\psi=\vartheta$ implies that $d\vartheta(F(\overline{x}))(dF(\overline{x})(w))\!<\!\infty$. Together with the properness of $d\vartheta(F(\overline{x}))$ and Lemma \ref{fsubderive-lemma}, $df(\overline{x})(w)$ is finite. Then, by Definition \ref{psderiv-def}, 
 \begin{equation}\label{p-sderive-equa1}
  {\rm dom}\,d^2\!f(\overline{x})(w\,|\,\cdot)\subset\mathcal{T}^2_{{\rm dom}\,f}(\overline{x},w).
 \end{equation}
 Fix any $z\in\mathbb{X}$. By Definition \ref{psderiv-def} and Lemma \ref{fsubderive-lemma}, there are sequences $\tau_k\downarrow 0$ and $z^k\!\to\! z$ such that 
 \[
   d^2\!f(\overline{x})(w|z)
  =\lim_{k\to\infty}\frac{\vartheta(F(\overline{x}+\tau_kw\!+\!\frac{1}{2}\tau_k^2z^k))-\vartheta(F(\overline{x}))\!-\!\tau_k d\vartheta(F(\overline{x}))(dF(\overline{x})(w))}{\tau_k^2/2}.
 \]
 Together with the parabolic semidifferentiability of $F$ at $\overline{x}$, it follows that
 \begin{align*}
  d^2\!f(\overline{x})(w|z)
  & =\lim_{k\to\infty}\frac{\vartheta(F(\overline{x})+\tau_k dF(\overline{x})w\!+\!\frac{1}{2}\tau_k^2 (F''(\overline{x};w,z)+o(\tau^2_k)/\tau_k^2))}{\tau_k^2/2} \\
  &\quad\qquad-\frac{\vartheta(F(\overline{x}))+\tau_k d\vartheta(F(\overline{x}))(dF(\overline{x})(w))}{\tau_k^2/2}\\
  &\ge d^2\vartheta(F(\overline{x}))(dF(\overline{x})(w)|F''(\overline{x};w,z))>-\infty,
 \end{align*}
 where the second inequality is due to Proposition \ref{psderive-PTD} (iii).
 Next we prove the converse inequality by two cases: $z\notin\!\mathcal{T}^2_{{\rm dom}\,f}(\overline{x},w)$ and $z\!\in\!\mathcal{T}^2_{{\rm dom}\,f}(\overline{x},w)$. Write $u=F''(\overline{x};w,z)$. 
 
 \noindent
 {\bf Case 1: $z\notin\mathcal{T}^2_{{\rm dom}\,f}(\overline{x},w)$}. By Corollary \ref{Stcone-domf}, $u\notin\mathcal{T}^2_{{\rm dom}\,\vartheta}(F(\overline{x}),dF(\overline{x})(w))$, which by Proposition \ref{psderive-PTD} (ii) for $\psi=\vartheta$ implies that $d^2\vartheta(F(\overline{x}))(dF(\overline{x})(w)|u)=\infty$. While from \eqref{p-sderive-equa1}, $d^2\!f(\overline{x})(w|z)=\infty$. Thus, 	$d^2\!f(\overline{x})(w|z)\!=d^2\vartheta(F(\overline{x}))(dF(\overline{x})(w)|d^2F(\overline{x})(w|z))$.
 	
 \noindent
 {\bf Case 2: $z\in\mathcal{T}^2_{{\rm dom}\,f}(\overline{x},w)$}. By Corollary \ref{Stcone-domf}, we have $u\in\mathcal{T}^2_{{\rm dom}\,\vartheta}(F(\overline{x}),dF(\overline{x})(w))$ and $z\in\mathcal{T}^{i,2}_{{\rm dom}\,f}(\overline{x},w)$. Pick any $\tau_k\downarrow 0$. From the definition of inner second-order tangent sets, there exists $z^k\to z$ such that for each $k\in\mathbb{N}$, ${\rm dom}f\ni x^k\!:=\overline{x}+\tau_k w+\frac{1}{2}\tau_k^2z^k$. Let $y^k\!:=F(\overline{x})\!+\!\tau_k dF(\overline{x})(w)\!+\!\frac{1}{2}\tau_k^2u$. By Proposition \ref{psderive-PTD} (ii)-(iii),
 \[
   \infty>d^2\vartheta(F(\overline{x}))(dF(\overline{x})(w)|u)
   =\lim_{k\to \infty}\frac{\vartheta(y^k)-\vartheta(F(\overline{x}))
    -\tau_kd\vartheta(F(\overline{x}))(dF(\overline{x})(w))}{\tau_k^2/2},
 \]
 which implies that $y^k\in{\rm dom}\,\vartheta$ for all sufficiently large $k$. Consequently, it holds that
 \begin{align}\label{ParaDerUb}
  d^2\!f(\overline{x})(w|z)
  &\le\liminf_{k\to\infty}\frac{f(x^k)-f(\overline{x})-\tau_k df(\overline{x})(w)}
 		{\tau_k^2/2} \nonumber\\
  &\le\limsup_{k\to\infty}\frac{\vartheta(F(x^k))-\vartheta(F(\overline{x}))-
 			\tau_kd\vartheta(F(\overline{x}))(dF(\overline{x})(w))}{\tau_k^2/2}\nonumber\\
  &=\lim_{k\to\infty}\frac{\vartheta(y^k)-\vartheta(F(\overline{x}))\!-\!
 			\tau_kd\vartheta(F(\overline{x}))(dF(\overline{x})(w))}{\tau_k^2/2}
 		+\limsup_{k\to \infty}\frac{\vartheta(F(x^k))\!-\!\vartheta(y^k)}{\tau_k^2/2}
 		\nonumber\\
  &=d^2\vartheta(F(\overline{x}))(dF(\overline{x})(w)|u)
  +\limsup_{k\to \infty}\frac{\vartheta(F(x^k))\!-\!\vartheta(y^k)}{\tau_k^2/2}
  \nonumber\\
  &\le d^2\vartheta(F(\overline{x}))(dF(\overline{x})(w)|u)
 		+L_{\vartheta}\limsup_{k\to \infty} \frac{2\|F(x^k)-y^k\|_2}{\tau_k^2}
 		\nonumber\ \ {\rm for\ some}\ L_{\vartheta}>0\\
  &= d^2\vartheta(F(\overline{x}))(dF(\overline{x})(w)\,|\,u),
  \end{align}
  where the third inequality is due to Assumption \ref{ass0} (i), the second equality is due to Proposition \ref{psderive-PTD} (iii), and the third one is by the parabolic semidifferentiability of $F$. 
 	
  The above arguments show that the first equality of part (i) holds. To achieve \eqref{fdomain-psderiv}, it suffices to prove that the converse inclusion in \eqref{p-sderive-equa1} holds. Pick any $z\in\mathcal{T}^2_{{\rm dom}\,f}(\overline{x},w)$. From Corollary \ref{Stcone-domf}, $u\in\mathcal{T}_{{\rm dom}\,\vartheta}^2(F(\overline{x}), dF(\overline{x})(w))$, which by Proposition \ref{psderive-PTD} (ii) implies that $d^2\vartheta(F(\overline{x}))(dF(\overline{x})(w)|u)\!<\!\infty$. Together with the first equality of part (i), we have $z\in{\rm dom}\,d^2\!f(\overline{x})(w|\cdot)$, so the converse inclusion in \eqref{p-sderive-equa1} follows.
 	
 \noindent
  {\bf(ii)} By Corollary \ref{Stcone-domf} and part (i), $\emptyset\ne\mathcal{T}^{i,2}_{{\rm dom}\,f}(\overline{x},w)=\mathcal{T}^2_{{\rm dom}\,f}(\overline{x},w)={\rm dom}\,d^2f(\overline{x})(w\,|\,\cdot)$. Consider any $z\in\mathbb{X}$. Pick any $\tau_k\downarrow 0$. From the discussions after Definition \ref{psderiv-def}, it suffices to argue that there exists a sequence $z^k\to z$ such that $\Delta_{\tau_k}^2f(\overline{x})(w|z^k)\to d^2f(\overline{x})(w|z)$ as $k\to\infty$. Indeed, when $z\notin\mathcal{T}^{i,2}_{{\rm dom}\,f}(\overline{x},w)={\rm dom}\,d^2f(\overline{x})(w\,|\,\cdot)$, by the definition of inner second-order tangent sets, for any $z^k\!\to\! z$,  $\overline{x}+\tau_kw+\frac{1}{2}\tau_k^2z^k\notin{\rm dom}\,f$. By the definition of the parabolic difference quotients of $f$ at $\overline{x}$ for $w$, it holds that
  \[
    \Delta_{\tau_k}^2f(\overline{x})(w|z^k)=\frac{f(\overline{x}+\tau_kw+\frac{1}{2}\tau_k^2z^k)-f(\overline{x})-\tau_kdf(\overline{x})(w)}{\tau_k^2/2}=\infty
    =d^2f(\overline{x})(w|z)
  \] 
  where the last equality is due to $z\notin{\rm dom}\,d^2f(\overline{x})(w\,|\,\cdot)$. When $z\in\mathcal{T}^{i,2}_{{\rm dom}\,f}(\overline{x},w)$, there exists $z^k\to z$ such that ${\rm dom}\,f\ni x^k=\overline{x}+\tau_kw+\frac{1}{2}\tau_k^2z^k$ for each $k$, and using the same arguments as those for the above \eqref{ParaDerUb} leads to 
  \begin{align*}
   \lim_{k\to\infty}\Delta_{\tau_k}^2f(\overline{x})(w|z^k)
   &\le\limsup_{k\to\infty}\frac{\vartheta(F(x^k))-\vartheta(F(\overline{x}))-
  	\tau_kd\vartheta(F(\overline{x}))(dF(\overline{x})(w))}{\tau_k^2/2}\\
   &\le d^2\vartheta(F(\overline{x}))(dF(\overline{x})(w)|F''(\overline{x};w,z))
   =d^2f(\overline{x})(w|z)
  \end{align*}
  where the equality is due to part (i). Note that $d^2f(\overline{x})(w|z)\le\lim_{k\to\infty}\Delta_{\tau_k}^2f(\overline{x})(w|z^k)$. 
  Then, $\lim_{k\to\infty}\Delta_{\tau_k}^2f(\overline{x})(w|z^k)\to d^2f(\overline{x})(w|z)$ as $k\to\infty$. The proof is completed. 
  \end{proof} 
 
  By the parabolic epi-differentiability of $f$, we can achieve its twice epi-differentiability at $\overline{x}\in{\rm dom}\,f$ for $\overline{v}\in\partial f(\overline{x})$ under the parabolic regularity assumption. Its proof is similar to that of \cite[Theorem 3.8]{MohS20-SIOPT}, and we include it for completeness.
 \begin{theorem}\label{theorem-tepi} 
 Fix any $\overline{x}\in{\rm dom}\,f$ and  $\overline{v}\in\partial\!f(\overline{x})$. Suppose that the MSQC holds for constraint system $F(x)\in{\rm dom}\,\vartheta$ at $\overline{x}$, that $\overline{v}\in\widehat{\partial}\!f(\overline{x})$, and that $f$ is parabolically regular at $\overline{x}$ for $\overline{v}$. Then, the function $f$ is properly twice epi-differentiable at $\overline{x}$ for $\overline{v}$ with 
 \begin{equation}\label{equa-theorem2}
   d^2f(\overline{x}|\overline{v})(w)
 	=\left\{\begin{array}{cl}
 	\min_{z\in\mathbb{X}}\big\{d^2\!f(\overline{x})(w|z)-\langle\overline{v},z\rangle\}&{\rm if}\ 
 	w\in\mathcal{C}_{\!f}(\overline{x},\overline{v}),\\
 	\infty &{\rm otherwise}.
  \end{array}\right.  
 \end{equation}
 \end{theorem} 	
 \begin{proof}
  Fix any $w\in\mathbb{X}$ and pick any $\tau_k\downarrow 0$. We proceed the arguments by two cases. 
 	
 \noindent
 {\bf Case 1:} $w\in\mathcal{C}_{\!f}(\overline{x},\overline{v})$. Now, from $\overline{v}\in\widehat{\partial}\!f(\overline{x})$ and Lemma \ref{ccone-ffun} (iii), it follows that ${\rm dom}\,d^2f(\overline{x}|\overline{v})=\mathcal{C}_{f}(\overline{x},\overline{v})$.  
 Since $f$ is parabolically regular at $\overline{x}$ for $\overline{v}$, by following the same arguments as those for the first part of the proof of \cite[Proposition 3.6]{MohS20-SIOPT}, there exists $\overline{z}\in{\rm dom}\,d^2\!f(\overline{x})(w\,|\,\cdot)$ such that
 \begin{equation}\label{equa1-critcone}  d^2f(\overline{x})(w|\overline{z})-\langle\overline{v},\overline{z}\rangle=d^2\!f(\overline{x}|\overline{v})(w).
 \end{equation}		
 Note that $w\in\mathcal{T}_{{\rm dom}\,f}(\overline{x})$. By Proposition \ref{prop-fpsubderive} (ii), $f$ is parabolically epi-differentiable at $\overline{x}$ for $w$, so we can find a sequence $z^k\to\overline{z}$ such that 
 \begin{align*}
  d^2\!f(\overline{x})(w|\overline{z})
  &=\lim_{k\to\infty}\frac{f(\overline{x}+\tau_kw+\frac{1}{2}\tau_k^2z^k)-f(\overline{x})-\tau_kdf(\overline{x})(w)}{\tau_k^2/2}\\
  &=\lim_{k\to\infty}\frac{f(\overline{x}+\tau_kw+\frac{1}{2}\tau_k^2z^k)-f(\overline{x})-\tau_k \langle \overline{v},w\rangle}{\tau_k^2/2}\\
  &=\lim_{k\to\infty}(\Delta_{\tau_k}^2f(\overline{x}|\overline{v})(w^k)+\langle\overline{v},z^k\rangle)\ \ {\rm with}\ w^k=w+\frac{1}{2}\tau_kz^k,
 \end{align*}
 where the second equality is due to $df(\overline{x})(w)=\langle\overline{v},w\rangle$ implied by $w\in \mathcal{C}_f(\overline{x},\overline{x})$. Together with \eqref{equa1-critcone}, it follows that $d^2\!f(\overline{x}|\overline{v})(w)=\lim_{k\to\infty}\Delta_{\tau_k}^2f(\overline{x}|\overline{v})(w^k)$.
 	
 \noindent
 {\bf Case 2:} $w\notin\mathcal{C}_{\!f}(\overline{x},\overline{v})$. Now from $\overline{v}\in\widehat{\partial}\!f(\overline{x})$ and Lemma \ref{ccone-ffun} (iii), $w\notin{\rm dom}\,d^2\!f(\overline{x}|\overline{v})$, i.e.,  $d^2\!f(\overline{x}|\overline{v})(w)=\infty$. Take $w^k=w$ for each $k\in\mathbb{N}$. By Definition \ref{ssderiv-def}, it holds that
 \[
  \infty=d^2\!f(\overline{x}|\overline{v})(w)\le\liminf_{k\to\infty}\Delta_{\tau_k}^2f(\overline{x}|\overline{v})(w^k)\le\infty,
  \]
 which shows that the sequence $w^k$ is such that		$d^2\!f(\overline{x}|\overline{v})(w)=\lim_{k\to\infty}\Delta_{\tau_k}^2f(\overline{x}|\overline{v})(w^k)$. 
 
 According to the discussion after Definition \ref{ssderiv-def}, the above arguments show that $f$ is twice epi-differentiable at $\overline{x}$ for $\overline{v}$. Moreover, equality \eqref{equa-theorem2} is implied by the above proof and \cite[Proposition 13.64]{RW98}, which in turn shows that $d^2f(\overline{x}|\overline{v})$ is proper.
 \end{proof} 

 Theorem \ref{theorem-tepi} shows that the parabolic regularity of $f$ is the key to the proper twice epi-differentiability of $f$. The next section is dedicated to this property.
 
 \section{Parabolic regularity of $f$}\label{sec5}

 First of all, we provide a convenient condition to identify the parabolic regularity of $f$. 
 \begin{theorem}\label{PropSuff}
  Fix any $\overline{x}\in{\rm dom}\,f$ and  $\overline{v}\in\partial\!f(\overline{x})$. Suppose that $\partial\vartheta(F(\overline{x}))=\widehat{\partial}\vartheta(F(\overline{x}))$ and the MSQC holds for system $F(x)\!\in\!{\rm dom}\,\vartheta$ at $\overline{x}$.   Then $f$ is parabolically regular at $\overline{x}$ for $\overline{v}$ if there is $\overline{\xi}\in\Lambda_{\overline{x},\overline{v}}\!:=\!\{\xi\in\partial\vartheta(F(\overline{x}))\,|\, \langle \xi,dF(\overline{x})(w')\rangle\ge\langle \overline{v},w'\rangle\ \forall w'\!\in\mathbb{X}\}$ such that
  \begin{equation}\label{suff-condition}
  \!\inf_{z\in\mathbb{R}^n}\!\big\{d(d\vartheta(F(\overline{x}))(dF(\overline{x})(w)))(F''(\overline{x};w,z))\!-\!\langle\overline{v},z\rangle\big\}
  \!=d^2(\overline{\xi}F)(\overline{x})(w)\ \ \forall w\in\mathcal{C}_{f}(\overline{x},\overline{v}).
 \end{equation}
 If  $\partial\vartheta(F(\overline{x}))\!=\!\widehat{\partial}\vartheta(F(\overline{x}))$ is replaced by the regularity of $\vartheta$ at $F(\overline{x})$, condition \eqref{suff-condition} becomes  
  \begin{equation}\label{suff-condition2}
  \!\inf_{z\in\mathbb{R}^n}\!\Big\{\sup_{u\in\mathcal{A}_{\vartheta}(F(\overline{x}),dF(\overline{x})(w))}\!\langle u,F''(\overline{x};w,z)\rangle-\langle\overline{v},z\rangle\Big\}
 		=d^2(\overline{\xi}F)(\overline{x})(w)\ \ \forall w\in\mathcal{C}_{f}(\overline{x},\overline{v}),
 \end{equation}
 where $\mathcal{A}_{\vartheta}(F(\overline{x}),dF(\overline{x})(w))\!=\!\big\{u\in\partial\vartheta(F(\overline{x}))\,|\,d\vartheta(F(\overline{x}))(dF(\overline{x})(w))=\langle u,dF(\overline{x})(w)\rangle\big\}$.  
 \end{theorem}
 \begin{proof}
 Fix any $w\in\mathcal{C}_{f}(\overline{x},\overline{v})$. We claim that the following two inequalities hold: 
 \begin{align}\label{Uboundineq}
  d^2\!f(\overline{x}|\overline{v})(w)
  &\le\inf_{z\in\mathbb{X}}
 	\big\{d^2\vartheta(F(\overline{x}))(dF(\overline{x})(w)|F''(\overline{x};w,z))
 		-\langle\overline{v},z\rangle\big\}<\infty,\\
 \label{Lboundineq}
  d^2\!f(\overline{x}|\overline{v})(w)
  &\ge\sup_{\xi\in\Lambda_{\overline{x},\overline{v}}}
  \big\{d^2\vartheta(F(\overline{x})|\xi)(dF(\overline{x})(w))+d^2(\xi F)(\overline{x})(w)\big\}.
  \end{align}
  Indeed, by Lemma \ref{ccone-ffun} (i), $w\in\mathcal{T}_{{\rm dom}\,f}(\overline{x})$; while by Proposition \ref{prop-fpsubderive} (i), for any $z\in\mathbb{X}$, $d^2\!f(\overline{x})(w|z)=d^2\vartheta(F(\overline{x}))(dF(\overline{x})(w)|F''(\overline{x};w,z))$. Along with \cite[Proposition 13.64]{RW98}, we get the first inequality in \eqref{Uboundineq}. From \eqref{fdomain-psderiv} and Corollary \ref{Stcone-domf}, ${\rm dom}\,d^2\!f(\overline{x})(w|\cdot)\ne\emptyset$, so there exists $z\in\mathbb{X}$ such that $d^2\vartheta(F(\overline{x}))(dF(\overline{x})(w)|F''(\overline{x};w,z))<\infty$. Thus, the second inequality in \eqref{Uboundineq} follows. To achieve inequality \eqref{Lboundineq}, pick any $\xi\in\Lambda_{\overline{x},\overline{v}}$. Recall that $F$ is semidifferentiable. From \cite[Theorem 7.21]{RW98}, it follows that for any $u\in \mathbb{X}$,
  \[
   dF(\overline{x})(u)=\lim_{\tau\downarrow 0,u'\to u}\Delta_{\tau}F(\overline{x})(u')\ \ {\rm with}\ \ \Delta_{\tau}F(\overline{x})(u'):=\tau^{-1}[F(\overline{x}+\tau u')-F(\overline{x})], 
  \]
  which implies that $\langle \xi, dF(\overline{x})(u)\rangle=d(\xi F)(\overline{x})(u)$ for any $u\in\mathbb{X}$. By the definition of the second-order difference quotients of $f$ at $\overline{x}$, for any $\tau>0$ and any $w'\in\mathbb{X}$, 
  \begin{align*}
   \Delta_{\tau}^2f(\overline{x}|\overline{v})(w')
   &=2\tau^{-2}[\vartheta(F(\overline{x}+\tau w'))-\vartheta(F(\overline{x}))
  	 -\tau\langle \overline{v},w'\rangle]\nonumber\\
   &\ge2\tau^{-2}[\vartheta(F(\overline{x})+\tau\Delta_{\tau}F(\overline{x})(w'))
 	-\vartheta(F(\overline{x}))-\tau d(\xi F)(\overline{x})(w')]\nonumber\\
   &=2\tau^{-2}[\vartheta(F(\overline{x})+\tau\Delta_{\tau}F(\overline{x})(w'))
 	-\vartheta(F(\overline{x}))-\tau\langle\xi,\Delta_{\tau}F(\overline{x})(w')\rangle]\nonumber\\
   &\quad+2\tau^{-2}[(\xi F)(\overline{x}+\tau w')-(\xi F)(\overline{x})
 		-\tau d(\xi F)(\overline{x})(w')]
  \end{align*}
 where the inequality is due to $\xi\in\Lambda_{\overline{x},\overline{v}}$ and $\langle \xi,dF(\overline{x})(w')\rangle=d(\xi F)(\overline{x})(w')$. From Definition \ref{ssderiv-def} and $\lim_{\tau\downarrow 0,w'\to w}\Delta_{\tau}F(\overline{x})(w')=dF(\overline{x})(w)$, it follows that 
 \[
  d^2\!f(\overline{x}|\overline{v})(w)\ge d^2\vartheta(F(\overline{x})|\xi)(dF(\overline{x})(w)) +d^2(\xi F)(\overline{x})(w).
 \]
 This, by the arbitrariness of $\xi$ in $\Lambda_{\overline{x},\overline{v}}$, implies that  inequality \eqref{Lboundineq} holds. 

 Now let $\overline{\xi}\in \Lambda_{\overline{x},\overline{v}}$ be such that \eqref{suff-condition} holds. As $\widehat{\partial}\vartheta(F(\overline{x}))=\partial\vartheta(F(\overline{x}))$,   from \eqref{Rsdiff-sderiv} we have $d\vartheta(F(\overline{x}))(dF(\overline{x})(w))\ge\langle\overline{\xi},dF(\overline{x})(w)\rangle$. While from $w\in\mathcal{C}_{f}(\overline{x},\overline{v})$ and Lemma \ref{ccone-ffun} (i), $   d\vartheta(F(\overline{x}))(dF(\overline{x})(w))=\langle\overline{v},w\rangle\le\langle\overline{\xi},dF(\overline{x})(w)\rangle$. These two inequalities imply that $dF(\overline{x})(w)\in\mathcal{C}_{\vartheta}(F(\overline{x}),\overline{\xi})$. 
 Invoking Proposition \ref{ssderive-PTD} (ii) with $\psi=\vartheta,y=F(\overline{x})$ leads to 
 \begin{align}\label{Lbound-scadineq}
  d^2\!f(\overline{x}|\overline{v})(w)
  &\ge d^2\vartheta(F(\overline{x})\,|\,\overline{\xi})(dF(\overline{x})(w))
 		+d^2(\overline{\xi}F)(\overline{x})(w)\nonumber\\
  & = d^2 \vartheta(F(\overline{x}))(dF(\overline{x})(w))
 		+d^2(\overline{\xi}F)(\overline{x})(w).
 \end{align}
 In addition, from inequality \eqref{Uboundineq} and Proposition \ref{psderive-PTD} (iv) with $\psi=\vartheta,y=F(\overline{x})$, 
 \begin{align}\label{Ubound-scadineq}
  &\!d^2\!f(\overline{x}|\overline{v})(w)
  \le\inf_{z\in\mathbb{R}^n}\big\{d^2\vartheta(F(\overline{x}))(dF(\overline{x})(w)|F''(\overline{x};w,z))-\langle\overline{v},z\rangle\} \nonumber\\
  &\quad=\! d^2\vartheta(F(\overline{x}))(dF(\overline{x})(w))
 	\!+\!\inf_{z\in\mathbb{R}^n}\!\big\{d(d\vartheta(F(\overline{x}))(dF(\overline{x})(w)))(F''(\overline{x};w,z))\!-\!\langle\overline{v},z\rangle\big\}.
 \end{align}
 Combining the above two inequalities with \eqref{suff-condition} and Proposition \ref{prop-fpsubderive} (i) yields that
 \begin{align*}
  d^2\!f(\overline{x}|\overline{v})(w)
  &= d^2 \vartheta(F(\overline{x}))(dF(\overline{x})(w))
 		+d^2(\overline{\xi}F)(\overline{x})(w)\\
  &=\inf_{z\in\mathbb{R}^n}\big\{d^2\vartheta(F(\overline{x}))(dF(\overline{x})(w)|F''(\overline{x};w,z))-\langle\overline{v},z\rangle\}\\
  &=\min_{z\in\mathbb{X}}\big\{d^2\!f(\overline{x})(w|z)-\langle\overline{v},z\rangle\},
  \end{align*}
 which implies that $f$ is parabolically regular at $\overline{x}$ for $\overline{v}$ by \cite[Proposition 13.64]{RW98}. When $\partial\vartheta(F(\overline{x}))\!=\!\widehat{\partial}\vartheta(F(\overline{x}))$ is replaced by the regularity of $\vartheta$ at $F(\overline{x})$, from inequality \eqref{Uboundineq} and Proposition \ref{psderive-PTD} (iv) with $\psi=\vartheta,y=F(\overline{x})$, equation \eqref{Ubound-scadineq} becomes the following one
  \begin{align*}\label{Ubound-scadineq}
  &d^2\!f(\overline{x}|\overline{v})(w)
  \le\inf_{z\in\mathbb{R}^n}\big\{d^2\vartheta(F(\overline{x}))(dF(\overline{x})(w)|F''(\overline{x};w,z))-\langle\overline{v},z\rangle\} \nonumber\\
  &=\!d^2\vartheta(F(\overline{x}))(dF(\overline{x})(w))
 	\!+\!\inf_{z\in\mathbb{R}^n}\!\Big\{\sup_{u\in\mathcal{A}_{\vartheta}(F(\overline{x}),dF(\overline{x})(w))}\!\langle u,F''(\overline{x};w,z)\rangle\!-\!\langle\overline{v},z\rangle\Big\}.
 \end{align*}
 Together with \eqref{Lbound-scadineq}, \eqref{suff-condition2} and Proposition \ref{prop-fpsubderive} (i), we obtain the conclusion. 
 \end{proof}

 By the proof of Theorem \ref{PropSuff}, $\overline{\xi}\in\mathop{\arg\max}_{\xi\in\Lambda_{\overline{x},\overline{v}}}d^2(\xi F)(\overline{x})(w)+d^2 \vartheta(F(\overline{x}))(dF(\overline{x})(w))$.
 The upper and lower estimates in \eqref{Uboundineq} and \eqref{Lboundineq} play a key role in the proof of Theorem \ref{PropSuff}, which extend the conclusion of \cite[Proposition 5.1]{MohS20-SIOPT} to the composition \eqref{ffun} involving a PWTD outer function and a parabolically semidifferentiable inner mapping. We notice that a general lower estimate was recently obtained in \cite[Theorem 4.1]{Benko22Arxiv} for the composition of the form \eqref{ffun} with an lsc outer function and a continuous $F$ that is calm at $\overline{x}$ in the direction of interest. The lower estimate in \eqref{Lboundineq} is more precise because the mapping $F$ is assumed to be parabolically semidifferentiable on $\mathcal{O}$ by Assumption \ref{ass0}, which is necessarily semidifferentiable by Definition 2.8 and then is stronger than the calmness of $F$ at $\overline{x}$ in a direction $w$.
  
 In the rest of this section, we confirm that condition \eqref{suff-condition2} holds for the following several classes of composite functions $f$, and then achieve its parabolic regularity and proper twice epi-differentiability by Theorems \ref{PropSuff} and \ref{theorem-tepi}, respectively: 
 \begin{itemize}
 \item[(I)] $F(x)\!:=(\|x_{J_1}\|_q,\ldots,\|x_{J_m}\|_q)^{\top}\,(q>1)$ for $x\in\mathbb{R}^n$ with a partition $\{J_1,\ldots,J_m\}$ of $[n]$, and $\vartheta(z)\!=\!\sum_{i=1}^m\rho_{\lambda}(z_i)\ (\lambda>0)$ for $z\in\mathbb{R}^m$ with  $\rho_{\lambda}\!:\mathbb{R}\to \mathbb{R}$ satisfying 
  \begin{itemize}
  \item[(C.1)] $\rho_{\lambda}$ is a PWTD function with $\rho_{\lambda}(0)=0$;
 		
  \item[(C.2)] $\rho_{\lambda}$ is differentiable at all $t\ne 0$ and  $\partial\rho_{\lambda}(0)=[-\lambda,\lambda]$; 
 		
  \item[(C.3)] $\rho_{\lambda}$ is regular and strictly continuous.
  \end{itemize}  
 \item[(II)] $F(x)\!=\|x_{2}\|_q-x_{1}\,(q>1)$ for $x=(x_1,x_2)\in\mathbb{R}\times\mathbb{R}^{n-1}$ and $\vartheta(t)=\delta_{\mathbb{R}_{-}}(t)$ for $t\in\mathbb{R}$. Now $f=\delta_{K}$ where  $K\!:=\{(x_1,x_2)\in\mathbb{R}\times\mathbb{R}^{n-1}\,|\,\|x_2\|_q\le x_{1}\}$ is known as the  $q$-order cone (see \cite{Xue00,Andersen02}). When $q=2$, $K$ is the popular second-order cone.

 \item[(III)] $F$ is twice differentiable on the open set $\mathcal{O}$ and $\vartheta$ satisfies Assumption \ref{ass0} (ii); 
  	
 \item[(IV)] $F(x)\!=\lambda_{\rm max}(x)$ for $x\in\mathbb{S}^n$ and  $\vartheta(t)=\delta_{\mathbb{R}_{-}}(t)$ for $t\in\mathbb{R}$, where $\mathbb{S}^n$ is the space of all $n\times n$ real symmetric matrices, endowed with the trace inner product $\langle\cdot,\cdot\rangle$, i.e. $\langle x,y\rangle={\rm tr}(x^{\top}y)$ for $x,y\in\mathbb{S}^n$, and its induced Frobenius norm $\|\cdot\|_F$, and $\lambda_{\rm max}(x)$ denotes the maximum eigenvalue of $x\in\mathbb{S}^n$. 	  
 \end{itemize}  
 The parabolic regularity and proper twice epi-differentiability of the composite functions of type III-IV are known, and we include them for illustrating the use of condition \eqref{suff-condition2}. 
 \subsection{Composite functions of type I}\label{sec4.1}

 This class of composite functions frequently appears in group sparsity optimization. Two popular examples for $\rho_{\lambda}$ are the following SCAD and MCP functions \cite{Fan01,Zhang10}:
 \begin{align*}
  \rho_{\lambda}(t)&:=\left\{\begin{array}{cl}
  \lambda|t| & {\rm if\;} |t|\le\lambda,\\
  \frac{-t^2+2a\lambda|t|-\lambda^2}{2(a-1)}& {\rm if\;} \lambda<|t|\le a\lambda,\\
  \frac{(a+1)}{2}\lambda^2 & {\rm if\;} |t|> a\lambda
  \end{array}\right.\ {\rm with}\ a>2;\\
 \rho_{\lambda}(t)&:=\lambda\, {\rm sign}(t)\!\int_{0}^{|t|}\Big(1-\frac{\omega}{\lambda b}\Big)_{+}d\omega\quad{\rm with}\ \ b>0.\nonumber
 \end{align*} 
 Let $F_i(z):=\|z\|_q$ for each $i\in[m]$ with $z\in\mathbb{R}^{J_i}$. Write ${\rm gs}(x)\!:=\{i\in[m]\ |\ x_{J_i}\ne 0\}$ and $\overline{\rm gs}(x)\!:=[m]\backslash{\rm gs}(x)$ for $x\in\mathbb{R}^n$. Fix any $x\in\mathbb{R}^n$ and $i\!\in\!{\rm gs}(x)$. The function $F_i$ is twice continuously differentiable at $x_{J_i}$ with  
 $\nabla F_i(x_{J_i})={\rm sign}(x_{J_i})\circ|x_{J_i}|^{q-1}\|x_{J_i}\|_q^{1-q}$,
 where $\circ$ is the Hadamard product operator. The mapping $F$ is Lipschitz continuous and directionally differentiable, and $dF(x)(w)\!=\!(dF_1(x_{J_1})(w_{J_1}),\ldots,dF_m(x_{J_m})(w_{J_m}))^{\top}$ for any $x,w\!\in\!\mathbb{R}^n$  with
 \begin{equation}\label{dF-equa}
  dF_i(x_{J_i})(w_{J_i})=\left\{\begin{array}{cl}
 		\langle\nabla\!F_i(x_{J_i}),w_{J_i}\rangle&{\rm if}\ i\in{\rm gs}(x),\\
 		\|w_{J_i}\|_q &{\rm otherwise} 	
 	\end{array}\right.\ \ {\rm for}\ i\in[m],
 \end{equation}
 and moreover, for any $\xi\in\mathbb{R}^m$, $d^2(\xi F)(x)(w)=\sum_{i=1}^md^2(\xi_iF_i)(x_{J_i})(w_{J_i})$ with
 \begin{equation}\label{d2F-equa0}
  d^2(\xi_iF_i)(x_{J_i})(w_{J_i})
  =\left\{\begin{array}{cl}
   \xi_i\langle w_{J_i},\nabla^2\!F_i(x_{J_i})w_{J_i}\rangle &{\rm if}\ i\in{\rm gs}(x),\\
   0 &{\rm otherwise}
  \end{array}\right.\ \ {\rm for}\ i\in[m].
 \end{equation}
 It is not difficult to check  the mapping $F$ is parabolically semidifferentiable, and at any $x,w\!\in\!\mathbb{R}^n$, 	$F''(x;w,z)\!=\!\big(F_1''(x_{J_1};w_{J_1},z_{J_1}),\ldots,F_m''(x_{J_m};w_{J_m},z_{J_m})\big)^{\top}$ for $z\in\mathbb{R}^n$ with 
 \begin{equation}\label{d2F-equa1}
  F_i''(x_{J_i};w_{J_i},z_{J_i})
   =\left\{\begin{array}{cl}
   \|z_{J_i}\|_q  & {\rm if \;} i\in\overline{\rm gs}(x)\cap\overline{\rm gs}(w),\\
 		\langle \nabla\!F_i(w_{J_i}),z_{J_i}\rangle & {\rm if \;} i\in\overline{\rm gs}(x)\cap{\rm gs}(w),\\
 		\!\langle w_{J_i},\nabla^2\!F_i(x_{J_i})w_{J_i}\rangle\!+\!\langle\nabla\!F_i(x_{J_i}),z_{J_i}\rangle & {\rm if \;} i\in{\rm gs}(x).
   \end{array}\right.
 \end{equation}
  By \cite[Theorem 8.30]{RW98}, conditions (C.2)-(C.3) imply that $d\rho_{\lambda}(0)(\omega)\!=\!\lambda|\omega|$ for any $\omega\!\in\!\mathbb{R}$. Then, from conditions (C.1)-(C.3), it is easy to obtain the following properties of $\vartheta$. 
 \begin{lemma}\label{scad-lemma}
  Let $\vartheta$ be given as in type I. Fix any $y$ with ${\rm supp}(y):=\{i\in [m]\,|\, y_i\neq 0\}$. 
  \begin{itemize}
  \item [(i)] For any $w\in\mathbb{R}^m$, $d\vartheta(y)(w)=\sum_{i\in{\rm supp}(y)}\rho_{\lambda}'(y_i)w_i+\lambda\sum_{i\notin{\rm supp}(y)}|w_i|$.
 		
  \item [(ii)] $\vartheta$ is regular at $y$ with $\emptyset\ne\widehat{\partial}\vartheta(y)=\partial \vartheta(y)
 		=\partial\rho_{\lambda}(y_1)\times\cdots\times\partial\rho_{\lambda}(y_m)$ where 
 		\[
 		\partial\rho_{\lambda}(y_i)
 		=\left\{\begin{array}{cl}
 			[-\lambda,\lambda]&\ {\rm if}\ i\notin{\rm supp}(y),\\
 			\{\rho_{\lambda}'(y_i)\} &\ {\rm if\;}\ i\in{\rm supp}(y)
 		\end{array}\right.\ \ {\rm for}\ i\in[m].
 		\]     
 \end{itemize}
 \end{lemma}

 By leveraging Lemma \ref{scad-lemma} (ii), we can characterize the subdifferential of $f$ as follows.
 \begin{lemma}\label{gsubdiff}
 Consider any $\overline{x}\in\mathbb{R}^n$ and any $v\in \partial\!f(\overline{x})$. Then, for each $i\in{\rm gs}(\overline{x})$, $v_i=\rho_{\lambda}'(\|\overline{x}_{J_i}\|_q)\nabla\!F_i(\overline{x}_{J_i})$; and for each $i\in\overline{{\rm gs}}(\overline{x})$, there exist $\eta_i\!\in\![-\lambda,\lambda]$ and $\zeta_i\in\mathbb{R}^{|J_i|}$ such that $v_i=\eta_i\zeta_i$ where, with $p:=\frac{q}{q-1}$, $\|\zeta_i\|_p\le 1$ if $\eta_i\ge 0$, and $\|\zeta_i\|_p=1$ if $\eta_i<0$.
 \end{lemma}
 \begin{proof}
  By the strict continuity of $\vartheta$ and $F$ and \cite[Theorem 10.49 \& Proposition 9.24 (b)]{RW98}, there exists $\eta\!\in\! \partial \vartheta(F(\overline{x}))$ such that $v_i\!\in\!\partial(\eta_iF_i)(\overline{x}_{J_i})$ for each $i\!\in\! [m]$. By Lemma \ref{scad-lemma} (ii), for each $i\!\in\!{\rm gs}(\overline{x})$, $\eta_i\!=\!\rho_{\lambda}'(\|\overline{x}_{J_i}\|_q)$, so $v_i\!=\!\rho_{\lambda}'(\|\overline{x}_{J_i}\|_q)\nabla\!F_i(\overline{x}_{J_i})$; and for each $i\in\overline{{\rm gs}}(\overline{x})$, $\eta_i\!\in\![-\lambda,\lambda]$. Fix any $i\!\in\!\overline{{\rm gs}}(\overline{x})$. When $\eta_i<0$, by \cite[Corollary 9.21]{RW98}, an elementary calculation yields that $\partial(\eta_iF_i)(\overline{x}_{J_i})\!=\!\partial(-|\eta_i|\|\cdot\|_q)(\overline{x}_{J_i})\!=\!\eta_i\{d\in\mathbb{R}^{|J_i|}\,|\,\|d\|_p\!=\!1\}$; and when $\eta_i\!\ge\! 0$, we have $\partial(\eta_iF_i)(\overline{x}_{J_i})\!=\!\eta_i\partial F_i(\overline{x}_{J_i})=\eta_i\{d\in\mathbb{R}^{|J_i|}\,|\,\|d\|_p\le 1\}$.
  \end{proof}

  Now we are ready to prove that condition \eqref{suff-condition2} holds for this class of functions $f$.
 \begin{proposition}\label{scad-prop}
  Fix any $\overline{x}\in \mathbb{R}^n$ and $\overline{v}\in\partial\!f(\overline{x})$. Then, condition \eqref{suff-condition2} holds, and consequently, $f$ is parabolically regular and properly twice epi-differentiable at $\overline{x}$ for $\overline{v}$.
 \end{proposition}
 \begin{proof}
  As ${\rm dom}\,\vartheta=\mathbb{R}^m$, the MSQC holds for system $F(x)\in{\rm dom}\,\vartheta$ at $\overline{x}$. Choose $\overline{\xi}\in\mathbb{R}^m$ with $\overline{\xi}_i=\rho_{\lambda}'(\|\overline{x}_{J_i}\|_q)$ for $i\in{\rm gs}(\overline{x})$ and $\overline{\xi}_i=\lambda$ for $i\in\overline{{\rm gs}}(\overline{x})$. By Lemma \ref{scad-lemma} (ii), obviously, $\overline{\xi}\in\partial\vartheta(F(\overline{x}))=\widehat{\partial}\vartheta(F(\overline{x}))$. We first argue that $\overline{\xi}\in\Lambda_{\overline{x},\overline{v}}$. Indeed, from $\overline{v}\in\partial\!f(\overline{x})$ and Lemma \ref{gsubdiff},  $\overline{v}_i\!=\rho_{\lambda}'(\|\overline{x}_{J_i}\|_q)\nabla\!F_i(\overline{x}_{J_i})$ for each $i\in{\rm gs}(\overline{x})$, and for each $i\in\overline{{\rm gs}}(\overline{x})$, there exist $\overline{\eta}_i\in[-\lambda,\lambda]$ and $\overline{\zeta}_i\in\mathbb{R}^{|J_i|}$ such that $\overline{v}_i=\overline{\eta}_i\overline{\zeta}_i$, where $\|\overline{\zeta}_i\|_p\le 1$ if $\overline{\eta}_i\ge 0$; otherwise $\|\overline{\zeta}_i\|_p=1$. Consequently, it is immediate to have that
  \begin{equation}\label{xibar-equa0}
  \langle \overline{v}, w'\rangle=\!\sum_{i\in{\rm gs}(\overline{x})}\rho_{\lambda}'(\|\overline{x}_{J_i}\|_q)\langle \nabla\!F_i(\overline{x}_{J_i}),w_{J_i}'\rangle+\!\sum_{i\in\overline{{\rm gs}}(\overline{x})}\overline{\eta}_i\langle\overline{\zeta}_i,w_{J_i}'\rangle\quad\ \forall w'\in\mathbb{R}^n.
  \end{equation}
  On the other hand, by equation \eqref{dF-equa} and Lemma \ref{scad-lemma} (i), it follows that for any $w'\!\in\!\mathbb{R}^n$, 
  \[
  \!\langle\overline{\xi}, dF(\overline{x})(w')\rangle
  =\!\!\sum_{i\in{\rm gs}(\overline{x})}\rho_{\lambda}'(F_i(\overline{x}_{J_i}))\langle \nabla\!F_i(\overline{x}_{J_i}),w_{J_i}'\rangle+\lambda\!\!\sum_{i\in\overline{{\rm gs}}(\overline{x})}\!\!\|w_{J_i}'\|_q\!=d\vartheta(F(\overline{x}))(dF(\overline{x})(w')). 
  \]
 The above two equations imply that $\langle \overline{\xi},dF(\overline{x})(w')\rangle\ge\langle \overline{v}, w'\rangle$ for all $w'\in\mathbb{R}^n$. Along with $\overline{\xi}\in\partial\vartheta(F(\overline{x}))$, we obtain $\overline{\xi}\in \Lambda_{\overline{x},\overline{v}}$. Next we prove that the vector $\overline{\xi}$ is such that condition \eqref{suff-condition2} holds. 
 Fix any $w\in\mathcal{C}_{f}(\overline{x},\overline{v})$. It suffices to show that 
 \begin{equation}\label{aim-equa}
  \Gamma_0\!:=\!\inf_{z\in\mathbb{R}^n}\!\Big\{\sup_{u\in\mathcal{A}_{\vartheta}(F(\overline{x}),dF(\overline{x})(w))}\!\langle u,F''(\overline{x};w,z)\rangle-\langle\overline{v},z\rangle\Big\}
 		=d^2(\overline{\xi}F)(\overline{x})(w).
 \end{equation}
 To this end, we first claim that  $u\!\in\!\mathcal{A}_{\vartheta}(F(\overline{x}),dF(\overline{x})(w))$ if and only if for each $i\!\in\![m]$, 
 \begin{equation}\label{urelation}
  u_i\in\left\{\begin{array}{cl}
 \{\rho_{\lambda}'(F_i(\overline{x}_{J_i}))\} &{\rm if}\ i\in{\rm gs}(\overline{x}),\\
  \{\lambda\} &{\rm if}\ i\in\overline{{\rm gs}}(\overline{x})\cap{\rm gs}(w),\\
 \ [-\lambda,\lambda]&{\rm if}\ i\in\overline{{\rm gs}}(\overline{x})\cap\overline{{\rm gs}}(w).   
 \end{array}\right.
 \end{equation}
 Indeed, by Lemma \ref{scad-lemma} (i) and equation \eqref{dF-equa}, $u\!\in\!\mathcal{A}_{\vartheta}(F(\overline{x}),dF(\overline{x})(w))$ if and only if
 \begin{align}\label{mid-equa51}
 &\sum_{i\in{\rm gs}(\overline{x})}u_i\langle \nabla\!F_i(\overline{x}_{J_i}),w_{J_i}\rangle+\sum_{i\in\overline{{\rm gs}}(\overline{x})}u_i\|w_{J_i}\|_q=\langle u,dF(\overline{x})(w)\rangle=d\vartheta(F(\overline{x}))(dF(\overline{x})(w))\nonumber\\
  &\qquad=\sum_{i\in{\rm gs}(\overline{x})}\rho_{\lambda}'(\|\overline{x}_{J_i}\|_q)\langle \nabla\!F_i(\overline{x}_{J_i}),w_{J_i}\rangle+\lambda\sum_{i\in\overline{{\rm gs}}(\overline{x})}\|w_{J_i}\|_q\ \ {\rm and}\ \ u\in\partial\vartheta(F(\overline{x})).
 \end{align}
 While by invoking Lemma \ref{scad-lemma} (ii), it is not hard to see that equation \eqref{mid-equa51} if and only if
 \begin{align}\label{mid-equa51}
  u_i=\rho_{\lambda}'(F_i(\overline{x}_{J_i}))\ {\rm for}\ i\in{\rm gs}(\overline{x}),\sum_{i\in\overline{{\rm gs}}(\overline{x})}\!u_i\|w_{J_i}\|_q=\lambda\!\sum_{i\in\overline{{\rm gs}}(\overline{x})}\!\|w_{J_i}\|_q,\,u\in\partial\vartheta(F(\overline{x})).
 \end{align}
 Combining Lemma \ref{scad-lemma} (ii) with the second equality and the inclusion in \eqref{mid-equa51}, we conclude that equation \eqref{mid-equa51} holds, or equivalently $u\!\in\!\mathcal{A}_{\vartheta}(F(\overline{x}),dF(\overline{x})(w))$, if and only if \eqref{urelation} holds for each $i\!\in\![m]$. 
 From $w\in\mathcal{C}_{f}(\overline{x},\overline{v})$, Lemmas \ref{ccone-ffun} (i) and \ref{scad-lemma} (i), and \eqref{dF-equa}, 
 \begin{equation*}
  \langle\overline{v},w\rangle=\sum_{i\in{\rm gs}(\overline{x})}\rho_{\lambda}'(\|\overline{x}_{J_i}\|_q)\langle\nabla\! F_i(x_{J_i}),w_{J_i}\rangle\!+\lambda\!\sum_{i\in\overline{{\rm gs}}(\overline{x})}\|w_{J_i}\|_q,
 \end{equation*} 
 which along with the above \eqref{xibar-equa0} for $w'=w$ implies that the following equality holds 
 \[
 \sum_{i\in\overline{{\rm gs}}(\overline{x})}\lambda\|w_{J_i}\|_q		\!=\!\sum_{i\in\overline{{\rm gs}}(\overline{x})}\overline{\eta}_i\langle\overline{\zeta}_i,w_{J_i}\rangle
 \Leftrightarrow
 \sum_{i\in\overline{{\rm gs}}(\overline{x})\cap{\rm gs}(w)}\!\!\lambda\|w_{J_i}\|_q	\!=\!\!\sum_{i\in\overline{{\rm gs}}(\overline{x})\cap{\rm gs}(w)}\!\!|\overline{\eta}_i|\langle{\rm sign}(\overline{\eta}_i)\overline{\zeta}_i,w_{J_i}\rangle.
 \]
 Note that $\langle u,\xi\rangle\le\|u\|_p\|\xi\|_q$ for any $u,\xi\in\mathbb{R}^{|J_i|}$, and the equality holds if and only if $\frac{u}{\|u\|_p}={\rm sign}(\xi)\circ|\xi|^{q-1}\|\xi\|_q^{1-q}$ for any $u\ne 0$ and $\xi\ne 0$. Then, for each $i\in\overline{{\rm gs}}(\overline{x})\cap{\rm gs}(w)$,
 \begin{equation}\label{temp-eta}
  {\rm sign}(\overline{\eta}_i)\overline{\zeta}_i={\rm sign}(w_{J_i})\circ|w_{J_i}|^{q-1}\|w_{J_i}\|_q^{1-q}=\nabla\!F_i(w_{J_i})\ \ {\rm and}\ \ |\overline{\eta}_i|=\lambda.
 \end{equation}
  For every $z\in\mathbb{R}^n$, from the equalities in \eqref{temp-eta} and equation \eqref{xibar-equa0} with $w'=z$, we have
 \begin{align*}
  \langle\overline{v},z\rangle
  &=\sum_{i\in{\rm gs}(\overline{x})}\rho_{\lambda}'(\|\overline{x}_{J_i}\|_q)\langle \nabla\!F_i(\overline{x}_{J_i}),z_{J_i}\rangle\!+\!\sum_{i\in\overline{{\rm gs}}(\overline{x})}\overline{\eta}_i\langle\overline{\zeta}_i,z_{J_i}\rangle \\
  &=\!\sum_{i\in{\rm gs}(\overline{x})}\rho_{\lambda}'(\|\overline{x}_{J_i}\|_q)\langle \nabla\!F_i(\overline{x}_{J_i}),z_{J_i}\rangle\!+\!\sum_{i\in\overline{{\rm gs}}(\overline{x})\cap\overline{{\rm gs}}(w)}\overline{\eta}_i\langle\overline{\zeta}_i,z_{J_i}\rangle\\
  &\qquad+\!\sum_{i\in\overline{{\rm gs}}(\overline{x})\cap{\rm gs}(w)}|\overline{\eta}_i|\langle{\rm sign}(\overline{\eta}_i)\overline{\zeta}_i,z_{J_i}\rangle\\
  &=\sum_{i\in{\rm gs}(\overline{x})}\rho_{\lambda}'(\|\overline{x}_{J_i}\|_q)\big[\langle \nabla\!F_i(\overline{x}_{J_i}),z_{J_i}\rangle +\langle w_{J_i},\nabla^2\!F_i(\overline{x}_{J_i})w_{J_i}\rangle\big]\\
 &\qquad+\!\sum_{i\in\overline{{\rm gs}}(\overline{x})\cap\overline{{\rm gs}}(w)}\overline{\eta}_i\langle\overline{\zeta}_i,z_{J_i}\rangle
  +\lambda\!\sum_{i\in\overline{{\rm gs}}(\overline{x})\cap{\rm gs}(w)}\langle\nabla F_i(w_{J_i}),z_{J_i}\rangle\\
 &\qquad -\sum_{i\in{\rm gs}(\overline{x})}\rho_{\lambda}'(\|\overline{x}_{J_i}\|_q)\langle w_{J_i},\nabla^2\!F_i(\overline{x}_{J_i})w_{J_i}\rangle.
 \end{align*} 
 Denote by $\Xi(z)$ the sum of the first three terms on the right-hand side. Then, 
 \begin{align}\label{final-ineq}
 \Gamma_0&=\inf_{z\in\mathbb{R}^n}	\sup_{u\in\mathcal{A}_{\vartheta}(F(\overline{x}),dF(\overline{x})(w))}\Big\{\langle u,F''(\overline{x};w,z)\rangle-\Xi(z)\Big\}\nonumber\\
 &\qquad+\sum_{i\in{\rm gs}(\overline{x})}\rho_{\lambda}'(\|\overline{x}_{J_i}\|_q)\langle w_{J_i},\nabla^2\!F_i(\overline{x}_{J_i})w_{J_i}\rangle\nonumber\\
 &=\inf_{z\in\mathbb{R}^n}	\sup_{u\in\mathcal{A}_{\vartheta}(F(\overline{x}),dF(\overline{x})(w))}\Big\{\langle u,F''(\overline{x};w,z)\rangle-\Xi(z)\Big\} +d^2(\overline{\xi}F)(\overline{x})(w),
 \end{align} 
 where the second equality is using equation \eqref{d2F-equa0} and $\overline{\xi}_i=\rho_{\lambda}'(\|\overline{x}_{J_i}\|_q)$ for each $i\in{\rm gs}(\overline{x})$. Now fix any $z\in\mathbb{R}^n$. For every $u\in\mathcal{A}_{\vartheta}(F(\overline{x}),dF(\overline{x})(w))$, by using \eqref{urelation} and \eqref{d2F-equa1} and comparing the expression of $\Xi(z)$ with that of $\langle u,F''(\overline{x};w,z)\rangle$, we conclude that
 \[
  \langle u,F''(\overline{x};w,z)\rangle-\Xi(z)=\!\sum_{i\in\overline{{\rm gs}}(\overline{x})\cap\overline{{\rm gs}}(w)}\big[u_i\|z_{J_i}\|_q-\overline{\eta}_i\langle\overline{\zeta}_i,z_{J_i}\rangle\big],
 \]
 which, along with the fact that $u\in\mathcal{A}_{\vartheta}(F(\overline{x}),dF(\overline{x})(w))$ iff \eqref{urelation} holds, implies that 
 \begin{align*}
  \sup_{u\in\mathcal{A}_{\vartheta}(F(\overline{x}),dF(\overline{x})(w))}\Big\{\langle u,F''(\overline{x};w,z)\rangle-\Xi(z)\Big\}
  =\sum_{i\in\overline{{\rm gs}}(\overline{x})\cap\overline{{\rm gs}}(w)}	\sup_{t\in[-\lambda,\lambda]}\Big\{t\|z_{J_i}\|_q-\overline{\eta}_i\langle\overline{\zeta}_i,z_{J_i}\rangle\Big\}.
 \end{align*}
 Together with the above equality \eqref{final-ineq}, it is immediate to obtain that  
 \begin{equation*}
  \Gamma_0=\sum_{i\in\overline{{\rm gs}}(\overline{x})\cap\overline{{\rm gs}}(w)}\inf_{z\in\mathbb{R}^{J_i}}	\sup_{t\in[-\lambda,\lambda]}\Big\{t\|z\|_q-\overline{\eta}_i\langle\overline{\zeta}_i,z\rangle\Big\}+ d^2(\overline{\xi}F)(\overline{x})(w). 
 \end{equation*}
 Note that $\inf_{z\in\mathbb{R}^{J_i}}\sup_{t\in[-\lambda,\lambda]}\big\{t\|z\|_q-\overline{\eta}_i\langle\overline{\zeta}_i,z\rangle\big\}\!\ge\! 0$ and the equality holds when $z\!=\!0$. This means that the desired \eqref{aim-equa} holds, so $f$ is parabolically regular at $\overline{x}$ for $\overline{v}$. Note that the function $\mathbb{R}^{l}\ni z\mapsto \rho_{\lambda}(\|z\|_q)$ is weakly convex because $\rho_{\lambda}$ is a nondecreasing weakly convex function, so is $f$, which implies that $\widehat{\partial}f(\overline{x})=\partial f(\overline{x})$. Then, $f$ is properly twice epi-differentiable at $\overline{x}$ for $\overline{v}$ by Theorem \ref{theorem-tepi}. The proof is completed.
 \end{proof}
 \subsection{Composite functions of type II}\label{sec4.2}
 \begin{proposition}\label{qcone-prop}
  Let $f=\delta_{K}$. Fix any $\overline{x}\in{\rm dom}\,f$ and $\overline{v}\in \partial\!f(\overline{x})$. Then, condition \eqref{suff-condition2} holds, so $f$ is parabolically regular and properly twice epi-differentiable at $\overline{x}$ for $\overline{v}$. 
  \end{proposition} 
 \begin{proof}
  Note that $f$ is a convex function due to $f=\delta_{K}$, and the MSQC holds for constraint system $F(x)\in{\rm dom}\,\vartheta$ at $\overline{x}$ because the Slater CQ holds. Fix any $w\in\mathcal{C}_{f}(\overline{x},\overline{v})$. Write $w=\!(w_1,w_2)\in\mathbb{R}\times\mathbb{R}^{n-1}$ and $\overline{v}=(\overline{v}_1,\overline{v}_2)\in\mathbb{R}\times\mathbb{R}^{n-1}$. By Proposition \ref{ccone-ffun} (i), $w\in\mathcal{T}_{{\rm dom}\,f}(\overline{x})$ and $d\vartheta(F(\overline{x}))(dF(\overline{x})(w))=\langle\overline{v},w\rangle$. Hence, $dF(\overline{x})(w)\in{\rm dom}\,d\vartheta(F(\overline{x}))=\!\mathcal{T}_{\mathbb{R}_-}(F(\overline{x}))$ where the equality is due to Lemma \ref{sderive-PTD} (ii). By invoking \cite[Theorem 8.2]{RW98}, 
  \begin{equation}\label{SconeCri}
  0=\delta_{\mathcal{T}_{\mathbb{R}_-}(F(\overline{x}))}(dF(\overline{x})(w))=d\vartheta(F(\overline{x}))(dF(\overline{x})(w))=\langle\overline{v},w\rangle.
  \end{equation}
  We proceed the proof by the following three cases: $\overline{x}\in{\rm int}\,K,\overline{x}=0$ and $\overline{x}\in{\rm bd}\,K\backslash\{0\}$. 
  
  \noindent	
  {\bf Case 1: $\overline{x}\in {\rm int}\,K$.} Now $F(\overline{x})<0$ and $\partial\! f(\overline{x})=\mathcal{N}_{K}(\overline{x})=\{0\}$. Then $\overline{v}=0,\,\partial \vartheta(F(\overline{x}))=\{0\},\mathcal{A}_{\vartheta}(F(\overline{x}),dF(\overline{x})(w))=\{0\}$ and $\Lambda_{\overline{x},\overline{v}}=\{0\}$. 
  Take $\overline{\xi}=0$. It is immediate to obtain 
   \begin{equation*}
  \!\inf_{z\in\mathbb{R}^n}\!\Big\{\sup_{u\in\mathcal{A}_{\vartheta}(F(\overline{x}),dF(\overline{x})(w))}\!\langle u,F''(\overline{x};w,z)\rangle-\langle\overline{v},z\rangle\Big\}=d^2(\overline{\xi}F)(\overline{x})(w)=0.
 \end{equation*}
 
  \noindent	
  {\bf Case 2: $\overline{x}=0$.} Now 	$F(\overline{x})=0,\partial\vartheta(F(\overline{x}))=\mathbb{R}_+$ and $dF(\overline{x})(w)\in \mathcal{T}_{\mathbb{R}_-}(F(\overline{x}))=\mathbb{R}_-$, so 
  \begin{equation}\label{Aset-equa}
  \mathcal{A}_{\vartheta}(F(\overline{x}),dF(\overline{x})(w))=\big\{u\in\mathbb{R}_{+}\,|\, u(\|w_2\|_q-w_1)=0\big\}.
  \end{equation}  
  In addition, $\partial f(\overline{x})= \mathcal{N}_{K}(\overline{x})=K^{\circ}\!=\{u=(u_1,u_2)\in\!\mathbb{R}\times\mathbb{R}^{n-1}\,|\,-\!u_1\ge\|u_2\|_p\}$, where $K^{\circ}$ denotes the negative polar cone of $K$. From $\overline{v}\in K^{\circ}$, we have $\|\overline{v}_2\|_p\leq -\overline{ v}_1$.  
  
  \noindent
  {\bf Subcase 2.1: $w=0$.} In this case, from the above equation \eqref{Aset-equa}, it follows that   
  \[
   \inf_{z\in\mathbb{R}^n}\!\Big\{\sup_{u\in\mathcal{A}_{\vartheta}(F(\overline{x}),dF(\overline{x})(w))}\!\langle u,F''(\overline{x};w,z)\rangle-\langle\overline{v},z\rangle\Big\}
   =\inf_{z\in\mathbb{R}^n}\sup_{u\ge 0}\big\{u(\|z_2\|_q-z_1)-\langle\overline{v},z\rangle\big\}
  \]  
  For any $z=(z_1,z_2)\in\mathbb{R}\times\mathbb{R}^{n-1}$, when $\|z_2\|_q>z_1$, we have $\sup_{u\ge 0}u(\|z_2\|_q-z_1)=\infty$; and when $\|z_2\|_q\le z_1$, we have $\sup_{u\ge 0}u(\|z_2\|_q-z_1)=0$. This implies that 
  \[
    \inf_{z\in\mathbb{R}^n}\sup_{u\ge 0}\big\{u(\|z_2\|_q-z_1)-\langle\overline{v},z\rangle\big\}
    =\inf_{\|z_2\|_q\le z_1}-\langle\overline{v},z\rangle=0,
  \]
  where the second equality is due to $-\langle\overline{v},z\rangle=-\overline{v}_1z_1-\overline{v}_2^{\top}z_2\ge\|\overline{v}_2\|_p(z_1-\|z_2\|_q)\ge 0$. 
 This means that $\inf_{z\in\mathbb{R}^n}\!\big\{\sup_{u\in\mathcal{A}_{\vartheta}(F(\overline{x}),dF(\overline{x})(w))}\!\langle u,F''(\overline{x};w,z)\rangle-\langle\overline{v},z\rangle\big\}=0$. While from \eqref{d2F-equa0}, $d^2(\xi F)(\overline{x})(w)=0$ for any $\xi\in\mathbb{R}^n$. Thus, for this case, condition \eqref{suff-condition2} holds.  
  
  \noindent
  {\bf Subcase 2.2: $w\ne 0$.} 
  From $\mathbb{R}_-\ni dF(\overline{x})(w)=\|w_2\|_q-w_1$, where the equality is due to \eqref{dF-equa}, we have $w_1\ge\|w_2\|_q$. When $w_1>\|w_2\|_q$, we must have $\overline{v}=0$ (if not, $\overline{v}_1<0$, and $0=\langle \overline{v},w\rangle=\overline{v}_1w_1+\overline{v}_2^{\top}w_2<\overline{v}_1\|w_2\|_q+\|\overline{v}_2\|_p\|w_2\|_q\le 0$, a contradiction), which along with $\mathcal{A}_{\vartheta}(F(\overline{x}),dF(\overline{x})(w))=\{0\}$ and $d^2(\xi F)(\overline{x})(w)=0$ for any $\xi\in\mathbb{R}^n$ implies that  
  \[
   \inf_{z\in\mathbb{R}^n}\!\Big\{\sup_{u\in\mathcal{A}_{\vartheta}(F(\overline{x}),dF(\overline{x})(w))}\!\langle u,F''(\overline{x};w,z)\rangle-\langle\overline{v},z\rangle\Big\}=0=d^2(\overline{\xi}F)(\overline{x})(w)\ \ {\rm for}\ \overline{\xi}=0.
  \]  
  This shows that condition \eqref{suff-condition2} holds. The rest only considers the case that $w_1=\|w_2\|_q$. In this case, $\mathcal{A}_{\vartheta}(F(\overline{x}),dF(\overline{x})(w))=\mathbb{R}_{+}$ follows \eqref{Aset-equa}. While from \eqref{d2F-equa1}, for any $u\in \mathbb{R}$ and $z\in\mathbb{R}^n$, $\langle u,F''(\overline{x};w,z)\rangle=u\langle \nabla F(w),z\rangle$. When $\overline{v}=0$, it holds that 
  \begin{align*}
   &\inf_{z\in\mathbb{R}^n}\!\Big\{\sup_{u\in\mathcal{A}_{\vartheta}(F(\overline{x}),dF(\overline{x})(w))}\!\langle u,F''(\overline{x};w,z)\rangle-\langle\overline{v},z\rangle\Big\}\\
   &=\inf_{z\in\mathbb{R}^n}\sup_{u\ge 0}\big\{u\langle \nabla F(w),z\rangle\big\}
   =\inf_{z\in \mathbb{R}^n}\{0\,|\, \langle \nabla F(w),z\rangle\leq 0\}=0,
  \end{align*}
  which along with $d^2(\xi F)(\overline{x})(w)=0$ for any $\xi\in\mathbb{R}^n$ shows that condition \eqref{suff-condition2} holds. Hence, it suffices to consider the case that $\overline{v}\ne 0$. Recall that $0=\langle\overline{v},w\rangle=\overline{v}_1w_1+\overline{v}_2^{\top}w_2\le \overline{v}_1w_1+\|\overline{v}_2\|_qw_1$, which implies that $\overline{v}_1\ge-\|\overline{v}_2\|_p$. Together with $\overline{v}_1\le-\|\overline{v}_2\|_p$, we have $\overline{v}_1=-\|\overline{v}_2\|_p$. Thus, $\overline{v}_2^{\top}w_2=\|\overline{v}_2\|_p\|w_2\|_q$. This, along with $0\ne \overline{v}_1=-\|\overline{v}_2\|_p$ and $0\ne w_1=\|w_2\|_q$, means that $\overline{v}_2=\|\overline{v}_2\|_p\eta=-\overline{v}_1\eta$ with $\eta:={\rm sign}(w_2)\circ|w_2|^{q-1} \|w_2\|_q^{1-q}$. Note that $\nabla F(w)=(-1;\eta)$. Clearly, $\overline{v}_1\nabla F(w)=-\overline{v}$ and $\overline{v}_1\le 0$. 
  In addition, from \eqref{d2F-equa1}, for any $z\in\mathbb{R}^n$, $F''(\overline{x};w,z)=\langle\nabla F(w),z\rangle$. Along with $\mathcal{A}_{\vartheta}(F(\overline{x}),dF(\overline{x})(w))=\mathbb{R}_{+}$, 
  \begin{align*}
   &\inf_{z\in\mathbb{R}^n}\!\Big\{\sup_{u\in\mathcal{A}_{\vartheta}(F(\overline{x}),dF(\overline{x})(w))}\!\langle u,F''(\overline{x};w,z)\rangle-\langle\overline{v},z\rangle\Big\}\\
   &=\inf_{z\in \mathbb{R}^n}\{-\langle\overline{v},z\rangle\,|\, \langle \nabla F(w),z\rangle\leq 0\}=\max_{t\le 0}\,\{0\ |\ t\nabla F(w)=-\overline{v}\}=0
  \end{align*}
  where the second equality is due to the dual of linear programs, and the third one is by the feasibility of $\overline{v}$. Since $d^2(\xi F)(\overline{x})(w)=0$ for any $\xi\in\mathbb{R}^n$, condition \eqref{suff-condition2} also holds. 
   
  \medskip
  \noindent 	
  {\bf Case 3: $\overline{x}\in {\rm bd}\,K\backslash\{0\}$.} Now  $\|\overline{x}_2\|_q=\overline{x}_1>0$. Clearly, $F$ is twice continuously differentiable in a neighborhood of $\overline{x}$. From \cite[Exercise 6.7]{RW98}, $\partial\!f(\overline{x})=\nabla F(\overline{x})\partial\vartheta(F(\overline{x}))$. Invoking Proposition \ref{SmoothResult} later with $\vartheta=\delta_{\mathbb{R}_{-}}$, we have condition \eqref{suff-condition2} holds for this case.  
  \end{proof}
  \begin{remark}
  When $q=2$, the parabolic regularity of $\delta_{K}$ is implied by the second-order regularity of $K$ by \cite[Propositions 3.103 \& 3.136]{BS00} and \cite[Lemma 15]{BS2005}, and Mohammadi et al. also provided a direct proof for this fact in \cite[Example 5.8]{MohMS21-TAMS} by using the developed chain rule for parabolic regularity and reformulating the second order cone as the smooth constraint system $(\|x_2\|^2\!-x^2_1,-x_1)\in \mathbb{R}^2_-$ for $(x_1,x_2)\in\!\mathbb{R}\times \mathbb{R}^{n-1}$.
  \end{remark}

  From Proposition \ref{qcone-prop}, we immediately obtain the parabolic regularity of the indicator of the Cartesian product of several $q$-order cones, which is stated as follows.  
  \begin{corollary}
  Let $K\!=K_1\times\cdots\times K_m$ with $K_i:=\{(x_{i1},x_{i2})\,|\,\|x_{i2}\|_q\le x_{i1}\}$ for each $i\in[m]$. Fix any $\overline{x}=(\overline{x}_1;\overline{x}_2;\ldots;\overline{x}_m)\in K$ and  $\overline{v}=(\overline{v}_1;\overline{v}_2;\ldots;\overline{v}_m)\in\mathcal{N}_{K}(\overline{x})$. Then, $\delta_{K}$ is parabolically regular at $\overline{x}$ for $\overline{v}$ and with $d^2\delta_{K}(\overline{x}|\overline{v})(w)=\sum_{i=1}^md^2\delta_{K_i}(\overline{x}_i|\overline{v}_i)(w_i)$ for each $w\in \mathcal{T}_{K}(\overline{x})\cap \{\overline{v}\}^\perp$.
  \end{corollary}
  \subsection{Composite functions $f$ of type III}\label{sec5.3}
      
  For this class of $f$, since $F$ is twice differentiable on the open set $\mathcal{O}$, at any $\overline{x}\in{\rm dom}\,f$, 
  \[
    dF(\overline{x})(w)=F'(\overline{x})w\ \ {\rm and}\ \ d^2(\xi F)(\overline{x})(w)
    =\langle \xi, \nabla^2F(\overline{x})(w,w)\rangle\quad\forall w\in\mathbb{X}, \xi\in\mathbb{R}^m,
  \]
  and then $\Lambda_{\overline{x},\overline{v}}$ in Proposition  \ref{PropSuff} is specified as 
  $\Lambda_{\overline{x},\overline{v}}=\{\xi\!\in\!\partial\vartheta(F(\overline{x}))
  \,|\, \nabla F(\overline{x})\xi\!=\!\overline{v}\}$.  
  \begin{proposition}\label{SmoothResult}
  Consider any $\overline{x}\in{\rm dom}\,f$ and any $\overline{v}\in\partial\!f(\overline{x})$. Suppose that the MSQC holds for constraint system $F(x)\in{\rm dom}\,\vartheta$, that $\partial\!f(\overline{x})\subset\nabla F(\overline{x})\partial\vartheta(F(\overline{x}))$, and that  $\vartheta$ is regular at $F(\overline{x})$. Then, condition \eqref{suff-condition2} holds, and consequently, $f$ is parabolically regular and properly twice epi-differentiable at $\overline{x}$ for $\overline{v}$.  
  \end{proposition}
 \begin{proof}
  As $\overline{v}\in\partial\!f(\overline{x})\subset\nabla F(\overline{x})\partial\vartheta(F(\overline{x}))$, there is $\overline{\xi}\in\partial\vartheta(F(\overline{x}))$ such that $\overline{v}=\nabla F(\overline{x})\overline{\xi}$. Fix any $w\in\mathcal{C}_{f}(\overline{x},\overline{v})$. By Lemma \ref{ccone-ffun} (i) and $\overline{v}=\nabla F(\overline{x})\overline{\xi}$, $F'(\overline{x})w\in\mathcal{C}_{\vartheta}(F(\overline{x}),\xi)$ for any $\xi\in\Lambda_{\overline{x},\overline{v}}$.  
  By Definition \ref{psemi-deriv}, $F''(\overline{x};w,z)=\nabla^2F(\overline{x})(w,w)\!+\!F'(\overline{x})z$ for $z\in\mathbb{X}$. To prove that condition \eqref{suff-condition2} holds, we introduce the following optimal value function
  \begin{equation}\label{value-fun}
  \Upsilon(p)\!:=\inf_{z\in\mathbb{X}}\sup_{u\in\mathcal{A}(F(\overline{x}),F'(\overline{x})w)}\Big\{\langle u,\nabla^2F(\overline{x})(w,w)\!+\!p\rangle+\langle\nabla\! F(\overline{x})u\!-\!\overline{v},z\rangle\Big\}.
  \end{equation} 
  Note that $\mathcal{A}_{\vartheta}(F(\overline{x}),F'(\overline{x})w)\!=\!\big\{\eta\in\partial\vartheta(F(\overline{x}))\,|\,d\vartheta(F(\overline{x}))(F'(\overline{x})w)=\langle \eta,F'(\overline{x})w\rangle\big\}$ is a closed convex set because $\partial\vartheta(F(\overline{x}))\!=\!\widehat{\partial}\vartheta(F(\overline{x}))$. Combining the definition of $\Upsilon$ in \eqref{value-fun} and Lemma \ref{DualProLem} in Appendix leads to the following inequality
  \begin{align*}
  \Upsilon(0)&\ge\!\sup_{\eta\in\mathbb{R}^m}\Big\{\langle \eta,\nabla^2F(\overline{x})(w,w)\rangle\ \ {\rm s.t.}\ \ \nabla F(\overline{x})\eta=\overline{v},\,\eta\in\mathcal{A}_{\vartheta}(F(\overline{x}),F'(\overline{x})w)\Big\}\\  	&=\sup_{\xi\in\Lambda_{\overline{x},\overline{v}}}\langle\xi,\nabla^2F(\overline{x})(w,w)\rangle,
  \end{align*}
  where the equality is due to $\Lambda_{\overline{x},\overline{v}}=\big\{\xi\in\mathbb{R}^m\,|\, \nabla F(\overline{x})\xi=\overline{v},\,\xi\in\mathcal{A}_{\vartheta}(F(\overline{x}),F'(\overline{x})w)\big\}$, and the inequality becomes an equality if $\partial\Upsilon(0)\neq\emptyset$. Hence, to prove that condition \eqref{suff-condition2} holds, we only need to argue that $\partial \Upsilon(0)\neq \emptyset$. Note that the difference between $\Upsilon(0)$ and the  infinimum in \eqref{Uboundineq} is the constant $d^2\vartheta(F(\overline{x}))(dF(\overline{x})(w))$, so $\Upsilon(0)<\infty$. Recall that $\Lambda_{\overline{x},\overline{v}}$ is nonempty. From the above inequality, $\Upsilon(0)>-\infty$. Thus, $\Upsilon(0)$ is finite. By \cite[Proposition 8.32]{RW98}, it suffices to argue that there exist $\varepsilon_1>0$ and $c_1>0$ such that 
  \begin{equation}\label{aimineq-Upsion}
   \Upsilon(p)\geq\Upsilon(0)-c_1\|p\|_2\quad{\rm for\ all}\ p\in\mathbb{R}^m\ {\rm with}\ \|p\|_2\le\varepsilon_1.
  \end{equation}  
  For this purpose, pick any small $\varepsilon>0$. Fix any $p\in\mathbb{R}^m$ with $\|p\|_2\le\varepsilon$. Invoking Proposition \ref{psderive-PTD} (iv) with $\psi=\vartheta$ and $y=F(\overline{x})$ immediately yields that  
  \begin{equation*}
  \Upsilon(p)+d^2\vartheta(F(\overline{x}))(F'(\overline{x})w)
  \!=\!\inf_{z\in\mathbb{X}}\big\{d^2\vartheta(F(\overline{x}))(F'(\overline{x})w\,|\,\nabla^2F(\overline{x})(w,w)+p\!+\!F'(\overline{x})z)-\langle\overline{v},z\rangle\big\}.
  \end{equation*}
  By Lemma \ref{ccone-ffun} (i), $w\in\mathcal{T}_{{\rm dom}\,f}(\overline{x})$, so $F'(\overline{x})w\in\mathcal{T}_{{\rm dom}\,\vartheta}(F(\overline{x}))$ follows Lemma \ref{tcone-domf}. Along with Proposition \ref{ssderive-PTD} (i) for $\psi=\vartheta,y=F(\overline{x})$, we conclude that $d^2\vartheta(F(\overline{x}))(F'(\overline{x})w)$ is finite. 
  If there is no $z\in\mathbb{X}$ such that $\nabla^2 F(\overline{x})(w,w)+p+F'(\overline{x})z\in\mathcal{T}^2_{{\rm dom}\vartheta}(F(\overline{x}),F'(\overline{x})w)$, by Proposition \ref{psderive-PTD} (ii) $d^2\vartheta(F(\overline{x}))(F'(\overline{x})w\,|\,\nabla^2F(\overline{x})(w,w)+p\!+\!F'(\overline{x})\,\cdot)\equiv\infty$, which along with the above equation and the finitness of $d^2\vartheta(F(\overline{x}))(F'(\overline{x})w)$ shows that $\Upsilon(p)=\infty$, and inequality \eqref{aimineq-Upsion} follows by taking $\varepsilon_1=\varepsilon$ and any $c_1>0$. Thus, it suffices to consider that there exists $z\in\mathbb{X}$ such that $\nabla^2 F(\overline{x})(w,w)+p+F'(\overline{x})z\in\mathcal{T}^2_{{\rm dom}\,\vartheta}(F(\overline{x}),F'(\overline{x})w)$, i.e., $\mathcal{S}_{w}(p)\ne\emptyset$, 
  where $\mathcal{S}_w$ is the multifunction defined in Lemma \ref{MSw-lemma} with $\varphi$ and $g$ replaced by $\vartheta$ and $F$, respectively. In other words, it holds that
  \begin{align}\label{Upsion-tempequa0}
  &\Upsilon(p)+d^2\vartheta(F(\overline{x}))(F'(\overline{x})w)\nonumber\\
  &=\inf_{z\in\mathcal{S}_{w}(p)}\!\big\{d^2\vartheta(F(\overline{x}))(F'(\overline{x})w\,|\,\nabla^2F(\overline{x})(w,w)+p\!+\!F'(\overline{x})z)-\langle\overline{v},z\rangle\big\}.
  \end{align}   
  Pick any $z\in\mathcal{S}_{w}(p)$. Let $u_p\!:=\!\nabla^2 F(\overline{x})(w,w)+p+\!F'(\overline{x})z$. Then $u_p\!\in\!\mathcal{T}^2_{{\rm dom}\,\vartheta}(F(\overline{x}),F'(\overline{x})w)$. 
  By Lemma \ref{MSw-lemma}, there exist $z^0\in\mathcal{S}_{w}(0)$ and $b\in \mathbb{B}_{\mathbb{R}^n}$, the unit ball centered at the origin of $\mathbb{R}^n$, such that $z=z^0+\kappa \|p\|_2b$ and $u_0=\nabla^2 F(\overline{x})(w,w)\!+\!F'(\overline{x})z^0\in\mathcal{T}^2_{{\rm dom}\vartheta}(F(\overline{x}),F'(\overline{x})w)$. Recall that $\mathcal{A}(F(\overline{x}),F'(\overline{x})(w))\ne\emptyset$. By Proposition \ref{psderive-PTD} (ii) and (iv), $d^2\vartheta( F(\overline{x}))(F'(\overline{x})w\,|\,\cdot)$ is a finite convex function on the set $\mathcal{T}^2_{{\rm dom}\,\vartheta}(F(\overline{x}),F'(\overline{x})w)$, and hence $d^2\vartheta(F(\overline{x}))(F'(\overline{x})w\,|\,\cdot)$ is strictly continuous relative to $\mathcal{T}^2_{{\rm dom}\,\vartheta}(F(\overline{x}),F'(\overline{x})w)$. Let $L_0$ be the Lipschitz constant of $d^2\vartheta( F(\overline{x}))(F'(\overline{x})w\,|\,\cdot)$ at $u_0$ relative to $\mathcal{T}^2_{{\rm dom}\,\vartheta}(F(\overline{x}),F'(\overline{x})w)$. From $z=z^0+\kappa \|p\|_2b$, if necessary by shrinking $\varepsilon$, it holds that 
  $\big|d^2\vartheta( F(\overline{x}))(F'(\overline{x})w\,|\,u_p)
  -d^2\vartheta( F(\overline{x}))(F'(\overline{x})w\,|\,u_0)\big|
  \le L_0\|u_p-u_0\|_2$.  Hence, 
  \begin{align*}
  &d^2\vartheta( F(\overline{x}))(F'(\overline{x})w\,|\,u_p)-\langle\overline{v},z\rangle
   -d^2\vartheta(F(\overline{x}))(F'(\overline{x})w)\\
  &\ge d^2\vartheta(F(\overline{x}))(F'(\overline{x})w\,|\,u_0)-d^2\vartheta(F(\overline{x}))(F'(\overline{x})w) \\
  &\quad-\langle\overline{v},z^0\rangle-L_0(\|p\|_2+\|F'(\overline{x})\|\|z-z^0\|)-\|\overline{v}\|\|z-z^0\|\\
  &\ge\Upsilon(0)-L_0(\|p\|_2+\kappa\|F'(\overline{x})\|\|p\|_2)-\kappa\|\overline{v}\|\|p\|_2,
  \end{align*}
  which along with \eqref{Upsion-tempequa0} implies that $\Upsilon(p)\ge \Upsilon(0)-[\kappa\|\overline{v}\|+L_0(1\!+\!\kappa\|F'(\overline{x})\|)]\|p\|_2$. By the arbitrariness of $p\in\mathbb{R}^m$ with $\|p\|_{2}\le\varepsilon$, inequality \eqref{aimineq-Upsion} holds with $\varepsilon_1=\varepsilon$ and  $c_1=\kappa\|\overline{v}\|+L_0(1\!+\!\kappa\|F'(\overline{x})\|)$. Thus, we complete the proof. 
 \end{proof}
 \subsection{Composite function $f$ of type IV}\label{sec4.4} 
 \begin{proposition}\label{SDP-theorem}
  Let $f$ be the composite function of type IV. Consider any $\overline{x}\in\mathbb{S}_{-}^n$ and any $\overline{v}\in\partial\!f(\overline{x})$. Then, condition \eqref{suff-condition2} holds, so $f$ is parabolically regular and properly twice epi-differentiable at $\overline{x}$ for $\overline{v}$. 
 \end{proposition}
 \begin{proof}
  Note that ${\rm dom}\,\vartheta=\mathbb{R}_-$ and $F$ is a convex function. The MSQC holds for system $F(x)\in{\rm dom}\,\vartheta$ at $\overline{x}$ because the Slater CQ holds. Fix any $w\in\mathcal{C}_{\!f}(\overline{x},\overline{v})$. By Lemma \ref{ccone-ffun} (i), $w\in\mathcal{T}_{{\rm dom}\,f}(\overline{x})$ and $d\vartheta(F(\overline{x}))(dF(\overline{x})(w))=\langle\overline{v},w\rangle$. The latter means that $dF(\overline{x})(w)\in{\rm dom}\,d\vartheta(F(\overline{x}))= \mathcal{T}_{\mathbb{R}_-}(F(\overline{x}))$, which by \cite[Theorem 8.2 (b)]{RW98} implies that 
  \begin{equation}\label{NconeCri}
  0=\delta_{\mathcal{T}_{\mathbb{R}_-}(F(\overline{x}))}(dF(\overline{x})(w))=d\vartheta(F(\overline{x}))(dF(\overline{x})(w))=\langle\overline{v},w\rangle.
  \end{equation}
  We proceed the arguments by the following two cases: $F(\overline{x})<0$ and $F(\overline{x})=0$. 
  
  \noindent	
  {\bf Case 1: $F(\overline{x})<0$.} Now  $\partial\! f(\overline{x})=\mathcal{N}_{\mathbb{S}^n_-}(\overline{x})=\{0\}$. Hence,  $\overline{v}=0$ and $\partial \vartheta(F(\overline{x}))=\{0\}$, which implies that $\Lambda_{\overline{x},\overline{v}}=\{0\}$ and $\mathcal{A}_{\vartheta}(F(\overline{x}),dF(\overline{x})(w))=\{0\}$.
  Consequently,  
  \begin{equation*}
  \!\inf_{z\in\mathbb{R}^n}\!\Big\{\sup_{u\in\mathcal{A}_{\vartheta}(F(\overline{x}),dF(\overline{x})(w))}\!\langle u,F''(\overline{x};w,z)\rangle-\langle\overline{v},z\rangle\Big\}=d^2(\overline{\xi}F)(\overline{x})(w)=0\ \ {\rm for}\ \overline{\xi}=0.
 \end{equation*}
  This shows that condition \eqref{suff-condition2} holds.
  
 \noindent	
 {\bf Case 2: $F(\overline{x})=0$.} Now $\partial \vartheta(F(\overline{x}))=\mathbb{R}_+$. From the above \eqref{NconeCri}, it follows that 
 \[
   \mathcal{A}_{\vartheta}(F(\overline{x}),dF(\overline{x})(w))=\big\{u\in\mathbb{R}_{+}\,|\, udF(\overline{x})(w)=0\big\}.
 \]
 By \cite[Theorem 10.49]{RW98}, $\partial\!f(\overline{x})=\{v\in t \partial F(\overline{x}) \,|\, t\in \mathbb{R}_+\}$.
Since $\overline{v}\in \partial f(\overline{x})$, there exist $\overline{t}\in\mathbb{R}_+$ and $\eta\in\partial F(\overline{x})$ such that $\overline{v}=\overline{t} \eta$.
 From $\eta\in\partial F(\overline{x})$ and the equivalence in \eqref{Rsdiff-sderiv}, $dF(\overline{x})(w)\ge \langle \eta,w\rangle$. Together with  $dF(\overline{x})(w)\in \mathcal{T}_{\mathbb{R}_-}(F(\overline{x}))=\mathbb{R}_{-}$ and \eqref{NconeCri}, we have $0\ge \overline{t}dF(\overline{x})(w)\ge \overline{t}\langle \eta, w\rangle=\langle\overline{v},w\rangle=0$, which implies that $\overline{t}dF(\overline{x})(w)=0$.
 
 \noindent
 {\bf Case 2.1: $dF(\overline{x})(w))\neq 0$.} Now $\mathcal{A}_{\vartheta}(F(\overline{x}),dF(\overline{x})(w))=\{0\}$ and $\overline{t}=0$, and the latter implies that $\overline{v}=0$. By \cite[Proposition 2.61]{BS00}, $dF(\overline{x})(w')\leq 0$ for all $w'\in \mathbb{S}^n$, so $\Lambda_{\overline{x},\overline{v}}=\{0\}$. Now using the same arguments as for Case 1 shows that condition \eqref{suff-condition2} holds.

 \noindent
 {\bf Case 2.2: $dF(\overline{x})(w))=0$.}
  Now $\mathcal{A}_{\vartheta}(F(\overline{x}),dF(\overline{x})(w))=\mathbb{R}_+$.   From $\eta\in\partial F(\overline{x})$ and the equivalence in \eqref{Rsdiff-sderiv}, $dF(\overline{x})(w')\ge \langle \eta,w'\rangle$ for any $w'\in\mathbb{S}^n$. Hence, $\overline{t}\in \Lambda_{\overline{x},\overline{v}}$. Notice that
  \begin{align*}
    &\!\inf_{z\in\mathbb{R}^n}\!\Big\{\sup_{u\in\mathcal{A}_{\vartheta}(F(\overline{x}),dF(\overline{x})(w))}\langle u,F''(\overline{x};w,z)\rangle-\langle\overline{v},z\rangle\Big\} \\
   &=\!\inf_{z\in\mathbb{R}^n}\!\Big\{\sup_{u\geq  0}\langle u,F''(\overline{x};w,z)\rangle-\langle\overline{v},z\rangle\Big\}
   =\!\inf_{z\in\mathbb{R}^n}\!\big\{-\!\langle\overline{v},z\rangle\,|\, F''(\overline{x};w,z)\leq 0\big\}\\
   &=-\sup_{z\in \mathcal{T}^2_{\mathbb{S}^n_-}(\overline{x},w)}\langle \overline{v}, z\rangle=-2\langle \overline{v},w\overline{x}^{\dagger}w\rangle=\overline{t} d^2F(\overline{x})(w),
  \end{align*}
  where the third equality is obtained by using Proposition \ref{SOTset} with $\varphi=\delta_{\mathbb{R}_{-}}$ and $g=F$, the fourth one is by \cite[Page 487]{BS00}, and the fifth one is by \cite[Corollary 5.13]{MohammadiSpectral}.
 This shows that condition \eqref{suff-condition2} holds for this case. The proof is completed.
 \end{proof}

%
%


\appendix
 \section{}
 \numberwithin{remark}{section}
 \begin{lemma}\label{DualProLem}
  Let $\mathcal{B}:\mathbb{X}\to \mathbb{R}^m$ be a linear mapping, and let $c\in \mathbb{R}^m$ and $\overline{v}\in \mathbb{X}$ be the given vectors. For every $p\in \mathbb{R}^m$, define
  \(	
  \upsilon(p)\!:=\inf_{z\in\mathbb{X}}\sup_{u\in \Omega}\big\{\langle u,\mathcal{B}z+c+p\rangle-\langle\overline{v},z\rangle\big\}
  \) 
  where $\Omega\subset\mathbb{R}^m$ is a closed convex set. Then, 
  \(
  \upsilon(0)\ge\sup_{u\in\Omega}\{\langle \xi,c\rangle\ {\rm s.t.}\ \mathcal{B}^*u=\overline{v}\}
 \) 
 and the equality holds if $\partial\upsilon(0)\ne\emptyset$, where $\mathcal{B}^*\!:\mathbb{R}^m\to\mathbb{X}$ is the adjoint of $\mathcal{B}$.
 \end{lemma}
 \begin{proof}
  For any $(z,p)\in\mathbb{X}\times\mathbb{R}^m$, define $\Psi(z,p):=\sup_{u\in \Omega}\{\langle u,\mathcal{B}z+c+p\rangle-\langle\overline{v},z\rangle\}$. Then, $\upsilon(p)=\inf_{z\in\mathbb{X}}\Psi(z,p)$, the value function of the perturbation of the convex program
  \begin{equation}\label{primal}
  \upsilon(0)=\inf_{z\in\mathbb{X}}\sup_{u\in \Omega}\big\{\langle u,\mathcal{B}z+c\rangle-\langle\overline{v},z\rangle\big\}. 
  \end{equation}
  From \cite[Section 2.5.1]{BS00}, the dual problem of \eqref{primal} takes the following form 
  \begin{equation}\label{dual}
   \sup_{\eta\in \mathbb{R}^m}\{-\Psi^*(0,\eta)\}. 
  \end{equation}  
  By the definition of conjugate function, it is not difficult to calculate that
  \begin{align*}
  \Psi^*(0,\eta)
  &=\sup_{(z,p)\in\mathbb{X}\times\mathbb{R}^m}\Big\{\langle \eta,p\rangle+\langle\overline{v},z\rangle- \sup\limits_{u\in \mathbb{X}}\{\langle u,\mathcal{B}z+c+p\rangle-\delta_{\Omega}(u)\}\Big\}\nonumber\\
  &=\sup_{(z,p)\in\mathbb{X}\times\mathbb{R}^m}\Big\{\langle \eta,p\rangle+\langle\overline{v},z\rangle-
  		\delta^*_{\Omega}(\mathcal{B}z+c+p)  \Big\}\nonumber\\
  &=\sup_{z\in\mathbb{X}}\Big\{\langle\overline{v},z \rangle+\sup\limits_{p'\in\mathbb{R}^m}\{\langle \eta,p'-c-\mathcal{B}z\rangle-
  		\delta^*_{\Omega}(p')\} \Big\}\nonumber\\
 &=\sup\limits_{z\in\mathbb{X}}\{\langle\overline{v}-\mathcal{B}^*\eta,z \rangle \}-\langle \eta,c\rangle+\delta_{\Omega}(\eta)
 =\delta_{\{0\}}(\mathcal{B}^*\eta-\overline{v})-\langle \eta,c\rangle+ 	\delta_{\Omega}(\eta).
 \end{align*}
 Then, from the weak duality theorem, $\upsilon(0)\ge\sup_{u\in\mathbb{R}^m}\{\langle u,c\rangle\ {\rm s.t.}\ \mathcal{B}^* u=\overline{v},\,u\in\Omega\}$. The inequality becomes an equality when there is no dual gap between \eqref{primal} and \eqref{dual}, 
  	which is guaranteed to hold under $\partial\upsilon(0)\ne\emptyset$ by \cite[Theorem 2.142 (i)]{BS00}. 
 \end{proof}
\end{document}